\documentclass{article}
\usepackage[english]{babel}
\usepackage{caption,amsmath,amssymb,graphicx,enumerate,xcolor,amsthm,url,geometry}

\newcommand{\tmaffiliation}[1]{\thanks{#1}}

\newtheorem{theorem}{Theorem}[section]
\newtheorem{lemma}[theorem]{Lemma}

\newtheorem{corollary}[theorem]{Corollary}
\newtheorem{proposition}[theorem]{Proposition}
\newtheorem{example}[theorem]{Example}
\newtheorem{remark}[theorem]{Remark}

\newtheorem{definition}[theorem]{Definition}
\def\bit{\begin{itemize}}
\def\eit{\end{itemize}}
\reversemarginpar   
\def\bc{\begin{center}}
\def\ec{\end{center}}
\def\bthm{\begin{theorem}}
\def\ethm{\end{theorem}}
\def\bcor{\begin{corollary}}
\def\ecor{\end{corollary}}
\def\bprop{\begin{proposition}}
\def\eprop{\end{proposition}}
\def\blem{\begin{lemma}}
\def\elem{\end{lemma}}

\def\brem{\begin{remark}}
\def\erem{\end{remark}}

\def\bdes{\begin{description}}
\def\edes{\end{description}}

\def\iti{\item[(i)]}
\def\itii{\item[(ii)]}

\def\beq{\begin{equation}}
\def\eeq{\end{equation}}

\def\ben{\begin{enumerate}}
\def\een{\end{enumerate}}

\def\beqar{\begin{eqnarray}}
\def\eeqar{\end{eqnarray}}
\def\beqarr{\begin{eqnarray*}}
\def\eeqarr{\end{eqnarray*}}

\def\RR{{\mathbb R}}  

\def\CC{{\mathbb C}}

\def\EE{{\mathbb E}}
\def\PP{{\mathbb P}}
\def\QQ{{\mathbb Q}}
\def\cA{\mathcal{A}}  
\def\cD{\mathcal{D}}  \def\cF{\mathcal{F}}
\def\cG{\mathcal{G}}  \def\cI{\mathcal{I}}

\def\cP{\mathcal{P}} \def\cQ{\mathcal{Q}} 
  
 \def\cW{\mathcal{W}}

\def\NN{{\mathbb N}}       

\def\one{{\bf 1}}

\def\eps{\epsilon}

\def\p{\varphi}

\def\part{\partial}

\def\d#1dt{\frac{d#1}{dt}}    

\begin{document}

\title{Stochastic flows and an interface SDE on metric graphs}

\author{Hatem Hajri
	\tmaffiliation{Universit\'e du Luxembourg, Email:
Hatem.Hajri.fn@gmail.com\newline Research supported by the National
Research Fund, Luxembourg, and cofunded under the Marie Curie
Actions of the European Comission (FP7-COFUND).}
  \and
  Olivier Raimond
  \tmaffiliation{Universit\'e Paris Ouest Nanterre La
D\'efense, Email: oraimond@u-paris10.fr}}

\maketitle

\begin{abstract}
This paper consists in the study of a stochastic differential equation on a metric graph, called an interface SDE $(\hbox{ISDE})$. 
To each edge of the graph is associated an independent white noise, which drives $(\hbox{ISDE})$ on this edge. 
This produces an interface at each vertex of the graph. We first do our study on star graphs with $N\ge 2$ rays. The case $N=2$ corresponds to the perturbed Tanaka's equation recently
studied by Prokaj \cite{MR18} and Le Jan-Raimond \cite{MR000} among others.
It is proved that $(\hbox{ISDE})$ has a unique in law solution, which is a Walsh's Brownian motion. This solution is strong if and only if $N=2$.

Solution flows are also considered. There is a (unique in law) coalescing stochastic flow of
mappings $\p$ solving $(\hbox{ISDE})$. For $N=2$, it is the only solution flow. For $N\ge 3$, $\p$ is not a strong solution and by filtering $\p$ with respect to the
family of white noises, we obtain a (Wiener) stochastic flow of kernels solution of $(\hbox{ISDE})$.
There are no other Wiener solutions.
Our previous results \cite{MR501011} in hand, these results are extended to more general metric graphs.

The proofs involve the study of $(X,Y)$ a Brownian motion in a two dimensional quadrant obliquely reflected at the boundary, with time dependent
angle of reflection. We prove in particular that, when $(X_0,Y_0)=(1,0)$ and if $S$ is the first time $X$ hits $0$, then $Y_S^2$ is a beta random variable of the second kind. We also calculate $\EE[L_{\sigma_0}]$, where $L$ is the local time accumulated at the boundary, and $\sigma_0$ is the first time $(X,Y)$ hits $(0,0)$.
\end{abstract}

\section{Introduction}
In \cite{MR18}, Prokaj proved that pathwise uniqueness holds for the perturbed Tanaka's equation 
\begin{equation}\label{PTSDE}
dX_t=\text{sgn}(X_t) dW^1_t+\lambda dW^2_t
\end{equation}
for all $\lambda\ne 0$, where $W^1$ and $W^2$ are two independent Brownian motions. 
When $\lambda=1$, after rescaling, 
setting $W^+=\frac{W^1+W^2}{\sqrt 2}$ and $W^-=\frac{W^2-W^1}{\sqrt 2}$, 
\eqref{PTSDE} rewrites
\begin{equation}\label{yaa}
dX_t=1_{\{X_t>0\}}dW^+_t+1_{\{X_t\le 0\}}dW^-_t.
\end{equation}
Using different techniques, the same result in the case of (\ref{yaa}) has also been obtained by Le Jan and Raimond \cite{MR000} (see also \cite{MR3055262,MR99905733}).

In this paper, an analogous SDE, called an Interface SDE (or $(\hbox{ISDE})$), on metric graphs is studied.
The first graphs we consider are star graphs with $N\ge 2$ rays $(E_i)_{1\le i\le N}$. 
Let $(W_i)_{1\le i\le N}$ be $N$ independent Brownian motions.
Then a solution to $(\hbox{ISDE})$ is a Walsh's Brownian motion that follows  $W^i$ during an excursion in $E_i$.
Note that $\eqref{yaa}$ is an $(\hbox{ISDE})$ on a star graph with $N=2$ rays.
Without Tsirelson's results \cite{MR1487755}, one could have the intuition that for $N\ge 3$, the situation is exactly the same as for (\ref{yaa}), i.e. $(\hbox{ISDE})$ has a unique strong solution. 
As this intuition is misleading, $(\hbox{ISDE})$ seems to be an interesting SDE.
In a second time, $(\hbox{ISDE})$ is extended to more general metric graphs with finite number of vertices and edges.
It is defined in the same way by attaching to the edges a family of independent Brownian motions. 

Stochastic flows (of mappings and kernels) solutions of $(\hbox{ISDE})$ are also studied. 
Now, we do not only follow the motion of one particle driven by $(\hbox{ISDE})$ but of a family of particles.
The law of a stochastic flow is determined by a consistent family of $n$-point motions (see \cite{MR2060298}).
These $n$-point motions consist of $n$ solutions of $(\hbox{ISDE})$, started from eventually $n$ different locations.
When $N=2$, Le Jan and Raimond \cite{MR000} proved that \eqref{yaa} generates a unique stochastic flow of mappings, which is coalescing in the sense that two solutions of \eqref{yaa} meet in finite time. 
When $N\ge 3$, we prove that there is still a stochastic coalescing flow of mappings, but it is no longer the unique stochastic flow solution of $(\hbox{ISDE})$.
 
To study $(\hbox{ISDE})$, we establish new results on obliquely reflected Brownian motions $(X,Y)$ on the positive quadrant.
These are of independent interest.
In particular, when  $(X_0,Y_0)=(x,0)$, we give the law of $Y_S$, where $S$ is the first time $X$ hits $0$.
We also give a necessary and sufficient condition on the angles of reflection for the integrability of the local time cumulated at the boundaries before $(X,Y)$ hits $(0,0)$.

\medskip
Let us now give some notations that will be used throughout this paper and recall the definition of stochastic flows.
\begin{itemize}
\item If $M$ is a locally compact metric space, $C_0(M)$ will denote the set of continuous functions on $M$ vanishing at $\infty$.
We will denote by $\mathcal{B}(M)$ the Borel $\sigma$-field on $M$ and by $\cP(M)$ the set of Borel probability measures on $M$.
\item A kernel on $M$ is a measurable mapping $K:M\to\cP(M)$. For $x\in M$, the probability measure $K(x)$ will also be denoted by $K(x,dy)$. We recall that two kernels $K_1$ and $K_2$ may be composed by the formula $K_1K_2(x,dz)=\int_{y\in M} K_1(x,dy)K_2(y,dz)$. 
\item The two-dimensional quadrant $[0,\infty[^2$ is denoted by $\mathcal Q$. Its boundary is $\partial \mathcal{Q}:= \partial_1\mathcal{Q}\cup\partial_2\mathcal{Q}$, where $\partial_1\mathcal Q=[0,\infty[\times \{0\}$ and $\partial_2\mathcal Q=\{0\}\times[0,\infty[$. 
We also set $\mathcal{Q}^*=\mathcal{Q}\setminus\{(0,0)\}$.
 \item For $X$ a continuous semimartingale, we will denote by $L_t(X)$ its
symmetric local time process at $0$, i.e.
$$L_t(X)=\lim_{\epsilon\rightarrow0}\frac{1}{2\epsilon}\int_{0}^{t}1_{\{|X_s|\le\epsilon\}}d\langle
X\rangle_s.$$
\end{itemize}

\paragraph{Stochastic flows\,:}
Let $M$ be a locally compact metric space and $(\Omega,\cA,\PP)$ be a probability space.

\begin{definition} \label{defsfm} 
A stochastic flow of mappings (SFM) $\p$ on $M$ is a random family $(\varphi_{s,t})_{s\le t}$ of measurable mappings on $M$ such that for all $(s,x)\in\RR\times M$, $\p_{s,s}(x)=x$ and
\begin{enumerate}
\item For all $h\in\RR$, $s\le t$, $\p_{s+h,t+h}$ is distributed like $\p_{s,t}$;
\item For all $s_1\le t_1\le\cdots \le s_n\le t_n$, the family $\{\p_{s_i,t_i}, 1\le i\le n\}$ is independent;
\item For all $s\le t\le u$ and all $x\in M$, a.s. $\p_{s,u}(x)=\p_{t,u}\circ\p_{s,t}(x)$;
\item The mapping $(s,t,x)\mapsto \p_{s,t}(x)$ is continuous in probability;
\item For all $s\le t$,  $\p_{s,t}(x)$ converges in probability towards $\infty$ as $x\to\infty$.
\end{enumerate}
\end{definition}

Stochastic flows of kernels (SFK's) $K$ on $M$ are defined as are defined SFM's, i.e. they are random families $(K_{s,t})_{s\le t}$ of measurable kernels on $M$ such that, for all $(s,x)\in\RR\times M$, $K_{s,s}(x)=\delta_x$ and satisfying 1.,2.,3.,4. and 5. in Definition \ref{defsfm} with $\p$ replaced by $K$ and $\p_{t,u}\circ\p_{s,t}$ replaced by $K_{s,t}K_{t,u}$.
We refer to Le Jan-Raimond \cite{MR2060298} for a more detailed definition and study.
Note that if $\varphi$ is a SFM, then $K$ defined by $K_{s,t}(x)=\delta_{\p_{s,t}(x)}$ is a SFK. When it is the case, by misuse of language, we will then say that $K$ is a SFM.  
All SFM's (resp. all SFK's) considered in the following will be measurable in the sense that the mapping $(s,t,x,\omega)\mapsto \p_{s,t}(x,\omega) \hbox{ (resp. $K_{s,t}(x,\omega)$)}$ is measurable.


\section{Main results}
We give here the results proved in this paper, and announced in the introduction.
\subsection{The interface SDE on a star graph}\label{interf}

\subsubsection{Walsh's Brownian motions on star graphs}\label{defwbm}
Let $G$ be a star graph with $N\ge 1$ rays $(E_i)_{1\le i\le N}$ and origin denoted by $0$, i.e.  $G=\cup_{i=1}^N E_i$ is a locally compact metric space such that
\begin{itemize}
\item $E_i\cap E_j=\{0\}$ if $i\neq j$, 
\item for all $i$, $E_i$ is isometric to $[0,\infty[$ via a mapping $e_i:[0,\infty[\to {E}_i$\end{itemize} 
and the metric $d$ on $G$ is such that
$d(x,y)=|v-u|$ if $x=e_i(u)$, $y=e_i(v)$ for some $i,u,v$ and $d(x,y)=d(x,0)+d(y,0)$ otherwise. 
Set for $x\in G$, $|x|=d(x,0)$, $G^*=G\setminus\{0\}$ and for $1\le i\le N$, $E_i^*=E_i\setminus\{0\}$.

\smallskip
Let us now define the Brownian motions on $G$, called Walsh's Brownian motions (abbreviated WBM's). 
These processes behave like a standard one-dimensional Brownian motion along a ray away from $0$ and, when they hit $0$, they randomly \textquotedblleft choose \textquotedblright\ a ray $E_i$.

Suppose given a family of parameters $(p_1,\dots,p_N) \in ]0,1[^N$ such that $\sum_ {i=1}^{N}p_i=1$.

\begin{definition} \cite{MR1022917} Let $(T^+_t,t\ge 0)$ be the semigroup associated to the reflecting Brownian motion on $\RR_+$ and let $(T^0_t,t\ge 0)$ be the semigroup associated to the Brownian motion on $\RR_+$ killed at $0$.
Then for $f\in C_0(G)$ and $x\in E_i$, denoting $f_j(r)=f\circ e_j(r)$ for $1\le j\le N$ and 
$\bar{f}(r)=\sum_{j=1}^N p_j f_j(r)$,
$$P_tf(x) = T_t^+ \bar{f}(|x|) + T^0_t(f_i-\bar{f})(|x|)$$
defines a Feller semigroup on $C_0(G)$ associated to a diffusion called a Walsh's Brownian motion on $G$.

If a filtration $(\mathcal{G}_t)_t$ is given, then $X$ is a $(\cG_t)_t$-WBM 
if $X$ is adapted to $(\cG_t)_t$ and if given $\cG_t$, $(X_{t+s},\,s\ge 0)$ 
is a WBM started at $X_t$.
\end{definition}

The family of parameters $(p_1,\dots,p_N)$ being fixed later on, all WBM's considered thereafter will be associated to this family of parameters.

\paragraph{The domain $\cD$\,:} 
We denote by $C_b^2(G^*)$ the set of all continuous
functions $f:G\to\mathbb{R}$ such that for all $i\in
[1,N]$, $f\circ e_i$ is $C^2$ on $]0,\infty[$ with bounded first
and second derivatives both with finite limits at $0$.
For $f\in C_b^2(G^*)$ and $x=e_i(r)\in G^{*}$, set $f'(x)=(f\circ
e_i)'(r)$, $f''(x)=(f\circ e_i)''(r)$.
When $x=0$, set $f'(0)=\sum_{i=1}^{N}p_i(f\circ e_i)'(0+)$ and
$f''(0)=\sum_{i=1}^{N}p_i(f\circ e_i)''(0+)$. Finally, set
\begin{equation}\label{domain}
\mathcal D=\big\{f\in C_b^2(G^*): f'(0)=0\big\}.
\end{equation}
Set $\mathcal D_0=\{f\in \mathcal D, f, f''\in C_0(G)\}$. Then $\cD_0$ is contained in the domain of the generator $A$ of $(P_t)_t$.
We have $Af=f''/2$ for $f\in\cD_0$ and $(A,\mathcal D_0)$ generates the Feller semigroup $(P_t)$ in a unique way. 
\subsubsection{The interface SDE}
\begin{definition}\label{eds}
A solution of the interface SDE $(\hbox{\hbox{\em ISDE}})$ on a star graph $G$ with $N\ge 1$ 
 rays $(E_i)_{1\le i\le N}$ is a
pair $(X,W)$ of processes defined on a filtered probability space
$(\Omega,(\mathcal F_t)_t,\mathbb P)$ such that
\begin{itemize}
\item[(i)] $W=(W^1,\dots,W^N)$ is a standard $(\mathcal
F_t)$-Brownian motion in $\mathbb R^N$;
\item[(ii)] $X$ is an $(\mathcal F_t)$-adapted continuous process on $G$;
\item[(iii)] For all $f\in \mathcal{D}$,
\begin{equation}\label{kdd00}
f(X_t)=f(X_0)+\sum_{i=1}^N \int_{0}^{t}f'(X_s)1_{\{X_s\in
E_i\}}dW^i_s+\frac{1}{2}\int_{0}^{t} f''(X_s)ds.
\end{equation}
\end{itemize}
We will say it is a strong solution if $X$ is adapted to the filtration $(\mathcal{F}^W_t)_t$.
\end{definition}
In the case where $N=1$, $(\hbox{ISDE})$ is easy to study. There is no interface and $(\hbox{ISDE})$ has a unique (strong) solution, which is a reflected Brownian motion. Thus in the following, we assume $N\ge 2$.
It can easily be seen (by choosing for each $i$ a function $f_i\in\mathcal{D}$ 
such that $f_i(x)=|x|$ if $x\in E_i$) that on $E_i$, away from $0$, 
$X$ follows the Brownian motion $W^i$. 

In the case $N=2$, one can assume $E_1=]-\infty,0]$ and $E_2=[0,\infty[$.
Applying It\^o-Tanaka's formula (or Theorem \ref{nh} below), we see that $(\hbox{ISDE})$ is equivalent to the skew Brownian motion version of (\ref{yaa}):
\begin{equation}\label{tgv0}
dX_t=1_{\{X_t>0\}}dW^+_t+1_{\{X_t\le 0\}}dW^-_t+(2p_1-1)dL_t(X).
\end{equation}
Note that when $p_1=1/2$, \eqref{yaa} and \eqref{tgv0} coincide. 
Our first result is the following
\bthm \label{propunic} For all $x\in G$, 
\begin{enumerate}
\item[(i)] There is a, unique in law, solution $(X,W)$ of $(\hbox{ISDE})$, with $X_0=x$. 
Moreover $X$ is a WBM.
\item[(ii)] The solution of $(\hbox{ISDE})$ is a strong solution if and only if $N=2$.
\end{enumerate}
\ethm

To prove (ii), we will check that pathwise uniqueness holds for $(\hbox{ISDE})$  when $N=2$.
This implies that the solution $(X,W)$ is a strong one.
The fact that for each $N\ge 3$, 
$(X,W)$ is not a strong solution is a consequence of a result of Tsirelson \cite{MR1487755}
(Theorem \ref{Ts} below) which states that if $N\ge 3$, there does not exist any $(\mathcal F_t)_t$-WBM on $G$ 
with $(\mathcal F_t)_t$ a Brownian filtration (see also \cite{MR1655299}).

\subsubsection{Stochastic flows solutions of (\hbox{ISDE})}

Let $(\Omega,\mathcal A,\PP)$ be a probability space. We define below what is meant by a flow solution of (\hbox{ISDE}). 
\begin{definition}\label{edsphi}
Let $\varphi$ be a SFM on $G$
and let $\mathcal W=(W^i, 1\le i\le N)$ be a family of independent real white
noises (see \cite[Definition 1.10]{MR2060298}). We say that $(\varphi,\mathcal W)$ solves $(\hbox{ISDE})$ if
for all $s\leq t$, $f\in \mathcal D$ and $x\in G$, a.s.
$$f(\varphi_{s,t}(x))=f(x)+\sum_{i=1}^{N}\int_s^t(1_{E_i}f')(\varphi_{s,u}(x))dW^i_u
+ \frac{1}{2}\int_s^tf''(\varphi_{s,u}(x))du.$$
We will say it is a Wiener solution if for all $s\le t$, 
$\mathcal F^{\varphi}_{s,t}\subset\mathcal F^{\mathcal W}_{s,t}$.
\end{definition}
It will be shown that if $(\varphi,\mathcal W)$ solves $(\hbox{ISDE})$,
then $\mathcal F^{\mathcal W}_{s,t}\subset\mathcal F^{\varphi}_{s,t}$ for all $s\le t$ and
thus we may just say $\varphi$ solves $(\hbox{ISDE})$. 
Note that if $\varphi$ is a Wiener solution, 
then $\mathcal F^{\varphi}_{s,t}=\mathcal F^{\mathcal W}_{s,t}$ for all $s\le t$.

We will prove the following
\begin{theorem}\label{thmphi}
\begin{enumerate} 
\item [(i)] There exists a SFM $\varphi$ solution of $(\hbox{ISDE})$.
This solution is unique in law. 
\item [(ii)] The SFM $\varphi$ is coalescing
in the sense that
for all $s\in\RR$ and $(x,y)\in G^2$, a.s.,
$$\inf\{t\ge s :\;
\varphi_{s,t}(x)=\varphi_{s,t}(y)\}\,<\,\infty.$$
\item [(iii)] The SFM $\varphi$ is a Wiener solution if and only if $N=2$.
\end{enumerate}
\end{theorem}
Note that (iii) in this theorem is a consequence of (ii) in Theorem \ref{propunic}.
We will also be interested in SFK's solving $(\hbox{ISDE})$ in the following sense.
\begin{definition}\label{edsSFK}
Let $K$ be a SFK on $G$ 
and $\mathcal W=(W^i, 1\le i\le N)$ be a family of independent real white noises.
We say that $(K,\mathcal W)$ solves $(\hbox{ISDE})$ if
for all $s\leq t$, $f\in \mathcal D$ and $x\in G$, a.s.
\begin{equation}\label{ker}
K_{s,t}f(x)=f(x)+\sum_{i=1}^{N}\int_s^t
K_{s,u}(1_{E_i}f')(x)dW^i_u + \frac{1}{2}\int_s^tK_{s,u}f''(x)du.
\end{equation}
We will say it is a Wiener solution if for all $s\le t$, 
$\mathcal F^{K}_{s,t}\subset\mathcal F^{\mathcal W}_{s,t}$.
\end{definition}
Since $\mathcal{F}^{\mathcal W}_{s,t}\subset \mathcal{F}^K_{s,t}$ for all $s\le t$, we may simply say that $K$ solves $(\hbox{ISDE})$.
Note that when $K=\delta_\varphi$, then $K$ solves $(\hbox{ISDE})$ if and only if $\varphi$ also solves $(\hbox{ISDE})$.

Let $(\varphi,\mathcal{W})$ be a solution to $(\hbox{ISDE})$, with $\p$ a SFM.
Then Lemma 3.2 in \cite{MR2060298} ensures that 
there is a SFK $K^{\cW}$ such that\,: 
 for all $s\le t$, $x\in G$, a.s. $K^{\mathcal W}_{s,t}(x)=E[\delta_{\varphi_{s,t}(x)}|\mathcal
F^{\mathcal W}_{s,t}].$
We have the following
\bprop\label{flow} $K^{\mathcal W}$ is the unique (up to modification) Wiener solution of $(\hbox{ISDE})$.
\eprop
Following \cite[Proposition 8]{MR2835247} or \cite[Proposition 3.1]{MR000}, this proposition can be proved by giving the Wiener chaos expansion of a Wiener solution (see also Proposition 4.2 in \cite{MRPOWS} for another proof).
A consequence of Proposition \ref{flow} and Theorem \ref{thmphi} (iii) is
\bcor\label{corw}
$K^{\mathcal{W}}$ is the only SFK solution of $(\hbox{ISDE})$ if and only if $N=2$.
\ecor
\begin{proof}
Assume $N=2$ and let $(K,\mathcal W)$ be a solution of $(\hbox{ISDE})$.  Then $\EE[K|\mathcal{W}]$ is a Wiener solution of $(\hbox{ISDE})$. By proposition \ref{flow} and Theorem  \ref{thmphi}-(iii), $\EE[K|\mathcal{W}]=\delta_\p$, where $\p$ is the SFM solution of $(\hbox{ISDE})$ given by Theorem  \ref{thmphi}-(i). This yields $K=\delta_\p$. For each $N\ge 3$, we have at least two distinct solutions\,: $\p$ from Theorem  \ref{thmphi}-(i) and $\EE[\delta_\p|\cW]$.
\end{proof}

In the case $N\ge 3$, the classification of all laws of flows solutions of $(\hbox{ISDE})$ is left open.

\subsection{Extension to metric graphs}\label{cont}

\subsubsection{Brownian motions on metric graphs}

A metric graph is a locally compact metric space for which there are
\begin{itemize}
\item a countable subset $V\subset G$ and
\item a countable family $\{E_i;\;i\in I\}$ of subsets of $G$ that are isometric to an interval $[0,L_i]$ or $[0,\infty[$
\end{itemize}
such that $G=\cup_{i\in I} E_i$, for all $i$, the boundary of $E_i$ is contained in $V$, and for all $i\neq j$, $E_i\cap E_j\subset V$.
The sets $V$ and $\{E_i;\;i\in I\}$ are respectively called the set of vertices and the set of edges of $G$. 
For $x$ and $y$ in $G$, $d(x,y)$ is defined as the length of a shortest continuous path joining $x$ and $y$. 
An edge is called a loop if its boundary is reduced to one point.

From now on, $G$ is a metric graph without loops for which $I$ (and thus $V$) is finite. 
Star graphs are examples of such metric graphs. 
For each $i\in I$, let $L_i$ be the length of the edge $E_i$.
Set $J_i=[0,L_i]$ when $L_i<\infty$ and $J_i=[0,\infty[$ when $L_i=\infty$.
Then there is an isometry $e_i:J_i\to {E}_i$.
When $L_i<\infty$, set $\{g_i,d_i\}=\{e_i(0),e_i(L_i)\}$ and when $L_i=\infty$, set $\{g_i,d_i\}=\{e_i(0),\infty\}$.
For each $v\in V$, set $I^+_v=\{i\in I;\; g_i=v\}$, $I^-_v=\{i\in I;\; d_i=v\}$ 
and $I_v=I^+_v\cup I^-_v$. 
Note that, since $G$ does not contain any loop, for all $v\in V$, $I_v^+\cap I_v^-=\emptyset$.
Denote by $N_v$ the cardinality of $I_v$.

To each $v\in V$, we associate a family of parameters $p^v:=\{p^v_i,\;i\in I_v\}$ such that $p_i^v\in ]0,1[$ and $\sum_{i\in I_v} p^v_i=1$. 
Let $G^{\ast}=G\setminus V$ and denote by $C^2_b(G^{\ast})$ the set of all continuous functions $f:G\to\mathbb{R}$ 
such that for all $i\in I$, $f\circ e_i$ is $C^2$ on the interior of $J_i$ and has bounded first and second derivatives both extendable by continuity to $J_i$.
For $f\in C^2_b(G^{\ast})$ and $x=e_i(r)\in G\backslash V$, 
set $f'(x)=(f\circ e_i)'(r)$, $f''(x)=(f\circ e_i)''(r)$ and for all $v\in V$, 
set $f'(v)=\bar{f'}(v)$ and $f''(v)=\bar{f''}(v)$ where for $g=f'$ or $f''$, $\bar{g}(v)$ is defined by
$$\bar{g}(v) = \sum_{i\in I_v^+}p_i^v (g\circ e_i)(0+) - \sum_{i\in I_v^-}p^v_i (g\circ e_i)(L_i-).$$

Finally set $${\mathcal D}=\big\{f\in C^2_b(G^{\ast}) : f'(v)=0\ \text{for all}\ v\in V\big\}.$$

Let $A$ be the operator defined on $\mathcal D_0:=\{f\in \mathcal D, f, f''\in C_0(G)\}$ by  $A f = f''/2$.
Theorem 2.1 in \cite{MR1743769} states that $(A,\cD_0)$ generates a (unique in law) Feller diffusion on $G$. This diffusion will be called a Brownian motion (BM) on $G$.

The family of parameters $p:=\{p^v,\;v\in V\}$ being fixed thereafter, all Brownian motions on $G$ will be associated to this family of parameters.

\subsubsection{The interface SDE on $G$}
The different notions of solutions of an interface SDE on $G$ are defined by replacing in Definitions \ref{eds}, \ref{edsphi} and \ref{edsSFK}
the set $\{1,\dots,N\}$ by $I$ and by taking for $\mathcal{D}$ the domain of functions defined above. 
Each of these SDE's will be denoted by $(\hbox{ISDE})$ and sometimes by $E(G,p)$, when we want to emphasize on $G$ and on the family of parameters $p$.

Note that if $(X,W)$ solves $(\hbox{ISDE})$, 
then up to the first hitting time of two different vertices, 
$(X,W)$ solves an SDE on a star graph. Using this observation and Theorem \ref{propunic}, 
one can prove that
\bthm \label{propunicext} For all $x\in G$, 
\begin{enumerate}
\item[(i)] There is a, unique in law, solution $(X,W)$ of $(\hbox{ISDE})$, with $X_0=x$. Moreover, $X$ is a BM on $G$.
\item[(ii)] The solution of $(\hbox{ISDE})$ is strong if and only if $N_v\le 2$ for all $v\in V$.
\end{enumerate}
\ethm

\subsubsection{Stochastic flows solutions of (\hbox{ISDE})}
Our purpose now is to construct  flows solutions of $(\hbox{ISDE})$. Our main tools will be Theorems 3.2 and 4.1 in \cite{MR501011}. Let us introduce some more notations. 

For each $v\in V$, set $G^v=\{v\}\cup\cup_{i\in I_v}E_i$. Then there is $\hat{G}^v$ a star graph 
and an isometry $i_v$ from $G^v$ onto a subset of $\hat{G}^v$ 
such that $i_v(E_i)\subset\hat{E}^v_i$ for all $i\in I_v$,
and where $\{\hat{E}^v_i;\,i\in I_v\}$ is the set of edges of $\hat{G}^v$.
Set, for all $v\in V$, $p^v:=(p_i^v\;:\; i\in I_v)$.

For $\mathcal W:=(W_i)_{i\in I}$, a family of independent white noises, and $v\in V$, set $\mathcal W^v:=(W^v_i)_{i\in I_v}$ the family of independent white noises defined by $W_i^v := W_i$ if $g_i=v$ and by  $W_i^v := - W_i$ otherwise.

A family of $\sigma$-fields $(\cF_{s,t};\;s\le t)$ will be said {\em independent on disjoint time intervals (abbreviated: i.d.i)} as soon as for all $(s_i,t_i)_{1\le i\le n}$ with $s_i\le t_i\le s_{i+1}$, the $\sigma$-fields $(\cF_{s_i,t_i})_{1\le i\le n}$ are independent.

Then, \cite[Theorem 3.2]{MR501011} states that to each family of flows $(\hat{K}^v)_{v\in V}$, and to each $\mathcal W:=(W_i)_{i\in I}$ a family of independent white noises such that
\begin{itemize}
\item[(i)] For all $v\in V$, $(\hat{K}^v,\mathcal{W}^v)$  is a solution of $E(\hat{G}^v,p^v)$ on $\hat{G}^v$;
\item[(ii)] The family of flows $(\hat{K}^v)_{v\in V}$ is i.d.i. in the sense that 
the family $\big(\vee_{v\in V} \mathcal F^{\hat{K}^v}_{s,t},\, s\le t\big)$ is i.d.i;
\end{itemize}
one can associate a (unique in law) solution $(K,\mathcal W)$ of $(\hbox{ISDE})$. 

Conversely \cite[Theorem 4.1]{MR501011} states that out of a solution $(K,\mathcal W)$ of $(\hbox{ISDE})$, 
one can construct a family of flows $(\hat{K}^v)_{v\in V}$ for which (i) and (ii) above are satisfied and such that the law of $(K,\mathcal{W})$ is uniquely determined by the law of this family. 
In the following, we will denote by $\hat{\PP}^v$ the law of the solution $(\hat{K}^v,\mathcal{W}^v)$.
Then $\hat{\PP}^v$ is a function of the law of $(K,\mathcal{W})$.

The idea behind these two results is that before passing through two distinct vertices, 
a \textquotedblleft global\textquotedblright\ flow solution of $(\hbox{ISDE})$ determines (and is determined by) 
a \textquotedblleft local\textquotedblright\ flow solution of an interface SDE on a star graph (associated to the vertex that has just been visited).

We will prove (see Theorem \ref{idi-ind}) in Section 5 that the i.d.i. condition implies conditional independence with respect to $\mathcal W$ of the flows $(\hat{K}^v)_{v\in V}$. This implies the following
\begin{theorem}\label{ann}
Each family $(\hat{\PP}^v)_{v\in V}$, with $\hat{\PP}^v$ the law of a solution of $E(\hat{G}^v,p^v)$, 
is associated to one and only one solution of $(\hbox{ISDE})$.
\end{theorem}
\begin{proof}
Suppose we are given $(\hat{\PP}^v)_{v\in V}$. Then, on some probability space, it is possible to construct a family of independent white noises $\mathcal W=(W^i, i\in I)$ 
and a family $(\hat{K}^v)_{v\in V}$ of SFK's respectively on $\hat{G}_v$ 
such that for all $v\in V$, $(\hat{K}^v,{\mathcal{W}}^v)$ is a solution of $E(\hat{G}^v,p^v)$ distributed as $\hat{\PP}^v$ and such that the flows $(\hat{K}^v)_{v\in V}$ are independent given $\mathcal{W}$. In other words,
$$\mathcal L\big((\hat{K}^v)_{v\in V}|\mathcal{W}\big)=\prod_{v\in V}\mathcal L\big(\hat{K}^v|\mathcal{W}^v\big)$$
where $\mathcal L$ stands for the conditional law. 
This implies in particular that the family $(\hat{K}^v)_{v\in V}$ is i.d.i., and Theorem 3.2 in \cite{MR501011} states that there exists $K$ a SFK on $G$ 
such that $(K,\mathcal{W})$ solves $(\hbox{ISDE})$, with $K$ obtained by well concatenating the flows $\hat{K}^v$. 

The fact that $(K,\mathcal W)$ is the only possible (in law) associated solution comes from the fact that the i.d.i. condition implies conditional independence. \end{proof}

This theorem and the results we obtained on star graphs imply
\begin{theorem}\label{thmphiext}
\begin{enumerate} 
\item [(i)] There is a, unique in law, SFM solution of $(\hbox{ISDE})$.
\item [(ii)]  A SFM solution of $(\hbox{ISDE})$ is a Wiener solution if and only if $N_v\le 2$ for all $v\in V$.
\item[(iii)] There is a unique (up to modification) SFK Wiener solution of $(\hbox{ISDE})$.
\end{enumerate}
\end{theorem}
\begin{proof}
Let $(K,\mathcal{W})$ be a solution of $(\hbox{ISDE})$, and denote by $(\hat{K}^v)_{v\in V}$ the associated family of flows, respective solutions of $E(\hat{G}^v,p^v)$.
Note that $K$ is a SFM if and only if the flows $\hat{K}^v$ are SFM's. Now (i) follows from Theorem \ref{ann} and Theorem \ref{thmphi} (i). 
Note also that $K$ is a Wiener solution if and only if the flows $\hat{K}^v$ are also Wiener solutions. Thus (ii) follows from Theorem \ref{thmphi} (iii), and (iii) follows from Theorem \ref{ann} and Theorem \ref{thmphi} (ii).
\end{proof}

Let us remark that if $(K,\mathcal{W})$ is a solution of $(\hbox{ISDE})$, 
then the law of $(K,\mathcal{W})$ depends on the choice of the isometries $(e_i)_{i\in I}$ which define the orientations of the edges of $G$. 
However the law of $K$ does not depend on this choice.

\subsection{Brownian motions with oblique reflections}
To prove Theorems \ref{propunic} and \ref{thmphi}, we shall study a Brownian motion in the quadrant $\cQ$,
obliquely reflected at the boundary and with time dependent angles of reflections. 
We now give an application of our methods to the obliquely reflected Brownian motion defined by Varadhan and Williams in
\cite{MR792398}.
Let first $(X,Y)$ be an obliquely reflected Brownian motion on $\RR\times \RR^+$, with angle of reflection $\theta\in]0,\frac{\pi}{2}[$, started at $(x,0)$ with $x>0$. Then there is $(B^1,B^2)$ a two dimensional Brownian motion such that
\begin{equation}\label{ornmhp}
\left\{\begin{array}{lll}
dX_t &=& dB^1_t  - \tan(\theta) dL_t(Y),\\
dY_t &=& dB^2_t +  dL_t(Y).
\end{array}\right.\end{equation}
The following proposition gives the law of $Y_S$, where $S$ is the first time the process $X$ hits $0$.

\begin{proposition}\label{law:Y_S}
The law of $Y^2_{S}/x^2$ is a Beta distribution of the second kind of parameters $\big(\frac{1}{2}-\frac{\theta}{\pi},\frac{1}{2} + \frac{\theta}{\pi}\big)$.
\end{proposition}

Fix now $\theta_1,\theta_2\in]0,\frac{\pi}{2}[$ and 
let $(X,Y)$ be an obliquely reflected Brownian motion in $\mathcal{Q}$ started at $(x,0)$ with $x>0$,
with angles of reflections on $\partial_1\mathcal Q$ and on $\partial_2\mathcal Q$ 
respectively given by $\theta_1$ and $\theta_2$, 
and killed at the first time $\sigma_0$ the process $(X,Y)$ hits $(0,0)$.

Define the sequence of stopping times $(S_k)_{k\ge 0}$ by $S_0=0$ and for $k\ge 0$,
\begin{eqnarray}
S_{2k+1} &=& \inf\{t\ge S_{2k} : \;X_t=0\},\nonumber\\
S_{2k+2} &=& \inf\{t\ge S_{2k+1} : \; Y_t=0\}.\nonumber\
\end{eqnarray}
Then $|{Z_{S_{2k}}}|=X_{S_{2k}}$, $|Z_{S_{2k+1}}|=Y_{S_{2k+1}}$ and Proposition \ref{law:Y_S} implies that
\begin{itemize}
\item[(i)] $(|Z_{S_{k+1}}|/|Z_{S_{k}}|)_{k\ge 0}$ is a sequence of independent random variables. 
\item[(ii)] For all $k\ge 0$, the law of $|Z_{S_{2k+1}}|^2/|{Z_{S_{2k}}}|^2$ 
(resp. $|Z_{S_{2k+2}}|^2/|{Z_{S_{2k+1}}}|^2$) is a Beta distribution of the second kind of parameters $\big(\frac{1}{2}-\frac{\theta_1}{\pi},\frac{1}{2} + \frac{\theta_1}{\pi}\big)$ (resp. $\big(\frac{1}{2}-\frac{\theta_2}{\pi},\frac{1}{2} + \frac{\theta_2}{\pi}\big)$).
\end{itemize}

Set $L_t:=L_t(X)+L_t(Y)$, the local time accumulated at $\partial \mathcal{Q}$.
It is known that $\sigma_0$ and $L_{\sigma_0}$ are finite
(see \cite{MR792398,MR779455}).
\begin{proposition}\label{yes}
We have that 
\begin{itemize}
\item If $\tan(\theta_1)\tan(\theta_2) > 1$, then $\mathbb{E}[L_{\sigma_0}]=\frac{x(\tan(\theta_2)+1)}{\tan(\theta_1)\tan(\theta_2)-1} \,<\, \infty.$
\item If $\tan(\theta_1)\tan(\theta_2) \le 1$, then $\mathbb{E}[L_{\sigma_0}] = \infty$.
\end{itemize}
\end{proposition}
The assumptions on the wedge and angles considered here are more suitable to our framework but our techniques may be applied to give an expression of $\EE[L_{\sigma_0}]$ in other
situations. 


\subsection{Outline of content}
In Section \ref{tdrbm}, obliquely reflected Brownian motions in $\mathcal Q$ are studied.
In Section \ref{prfthm12},  Theorem \ref{propunic} is proved. 
In Section \ref{S3}, using in particular the results of Section \ref{tdrbm}, Theorem \ref{thmphi} (i) and (ii) are proved. 
In Section \ref{last1}, we prove that the i.d.i. condition implies conditional independence thus completing the proof of Theorem \ref{ann}.
Finally in Section \ref{final}, some extensions are discussed.

\section{Brownian motion in the quadrant with time dependent angles of
reflection} \label{tdrbm}

\subsection{Brownian motion on the half-plane with oblique
reflection} \label{orbmhp}
Fix $\theta\in ]0,\pi/2[$. Let $Z=(X,Y)$ be the process
started from $(x,y)$ in $\RR\times \RR_+$ obliquely reflected at
$\{y=0\}$, with angle of reflection $\theta$. More
precisely, $Z$ satisfies (\ref{ornmhp}) with $(B^1,B^2)$ a two dimensional Brownian motion. 
Set $S=\inf\{ s:\; X_s=0\}$. When $y=0$ and $x>0$, $Z_s\in\cQ$ for all $s\le S$, and we denote by $\PP^\theta_x$ the law of $(Z_s;\;s\le S)$.
Observe that we have the following scaling property\,:
\bprop\label{scalprop}
For all $x>0$, if the law of $(Z_s;\;s\le S)$ is $\PP^\theta_1$,
then the law of $(xZ_{x^{-2}s};\;s\le x^{2}S)$ is $\PP^\theta_x$.
\eprop

For $z\in \CC$, $\hbox{arg}(z)$, $\mathcal{R}(z)$ and
$\mathcal{I}(z)$ will denote respectively the argument, the real
part and the imaginary part of $z$. A complex $z=x+iy$ will be identified with $(x,y)\in\RR^2$. So, if $f$ is some mapping on $\CC$, we will write as well $f(z)$ or $f(x,y)$.
Following \cite{MR792398}, if $f$ is an holomorphic function on an open set $U$ containing $\mathcal Q^{\ast}$
such that $f(z)\in\mathbb R$ for all $z\in]0,\infty[$, then $\phi(z):=\mathcal{R}\big(f(z)e^{-i\theta}\big)$  is harmonic on $U$.
Moreover, 
\begin{equation}
\label{bord}v_1(\theta).\nabla \phi (x,0)=0\quad \hbox{ for  $x>0$ and where $v_1(\theta)=(-\text{tan}(\theta),1)$}.\end{equation}
Indeed, the condition
$f(z)\in\mathbb R$ for all $z\in]0,\infty[$ implies that
$f'(z)\in\mathbb R$ for all $z\in]0,\infty[$. Thus
$$\nabla
\phi(x,0)=\big(\mathcal{R}\big(f'(x)e^{-i\theta}\big),\mathcal{R}\big(if'(x)e^{-i\theta}\big)\big)=f'(x)(\cos(\theta),\sin(\theta))$$
and \eqref{bord} follows.
These properties imply in particular that $(\phi(Z_{s\wedge S}))_s$ is
a local martingale. For $b\in\mathbb R$ and $f(z)=z^b$, the
function $\phi$ defined above will be denoted by $\phi_b$.

\blem\label{rtt} Let $(Z_s;\;s\le S)$ be a process of law $\PP^\theta_x$, with $x>0$. 
\begin{itemize}
\iti For $a>x$ and  $0<b< 1+2 \theta/\pi$, set $c_b=1$ if $b\le 4\theta/\pi$ and
$c_b=\cos(\theta)/\cos(b\pi/2-\theta)$ otherwise. Then
$$\mathbb P\left(\sup_{s\le S}|Z_s|>a\right)\le c_b
\left(\frac{x}{a}\right)^{b}.$$
\itii
For $a<x$ and  $0<b< 1-2 \theta/\pi$, set $c_b=\cos(\theta)/\cos(b\pi/2+\theta)$. Then
$$\mathbb P\left(\inf_{s\le S}|Z_s|<a\right)\le c_b
\left(\frac{a}{x}\right)^{b}.$$
\end{itemize}
\elem

\begin{proof}
Using the scaling property, we may take $x=1$. 
For $a \ge 0$, set $\sigma_a=\inf\{s:\, |Z_s|=a\}$.\\
{\bf Proof of (i):} Fix $a>1$ and $0<b<1+2\theta/\pi$. For
$c^0_b=\inf\{\cos(\theta),\cos(b \pi/2-\theta)\}$ and $s\le S$,
we have
$$c^0_b |Z_s|^b \le \phi_b(Z_s) \le |Z_s|^b.$$
Moreover 
$$\mathbb P(\sup_{s\le S}|Z_s|>a)=\mathbb P(\sigma_a<S).$$
Since $(\phi_b(Z_{s\wedge \sigma_a\wedge S}))_s$ is a true martingale, for all $s\ge 0$, 
$$\cos(\theta)=\phi_b(1)=\EE[\phi_b(Z_{s\wedge \sigma_a\wedge
S})]$$
which is larger than
$$\EE[\phi_b(Z_{s\wedge \sigma_a}) 1_{\{\sigma_a<S\}}].$$
As $s\to\infty$, this last term converges using dominated
convergence to
$$\EE[\phi_b(Z_{\sigma_a}) 1_{\{\sigma_a<S\}}] \ge c^0_b a^b
\mathbb P(\sigma_a<S).$$
This easily implies (i).\\
{\bf Proof of (ii):}
Fix $a<1$ and $0<b<1-2\theta/\pi$. For $c^1_b=\cos(b
\pi/2+\theta)$ and $t\le S$,
$$c^1_b |Z_s|^{-b} \le \phi_{-b}(Z_s) \le |Z_s|^{-b}.$$
We also have that
$$\mathbb P(\inf_{s\le S}|Z_s|<a)=\mathbb P(\sigma_a<S).$$
By the martingale property, for all $s\ge 0$,
$$\cos(\theta)=\phi_{-b}(1)=\EE[\phi_{-b}(Z_{s\wedge
\sigma_a\wedge S})]$$
which is larger than
$$\EE[\phi_{-b}(Z_{t\wedge \sigma_a}) 1_{\{\sigma_a<S\}}]$$
and this converges as $s\to\infty$ to 
$$\EE[\phi_{-b}(Z_{\sigma_a}) 1_{\{\sigma_a<S\}}] \ge c^1_b
a^{-b} \mathbb P(\sigma_a<S).$$
This easily implies (ii).
\end{proof}

\bcor \label{cor23} Let $(Z_s;\;s\le S)$ be distributed as
$\PP^\theta_x$. Then, if $|b-2\theta/\pi|< 1$,
$$\EE(\sup_{s\le S}|Z_s|^b)<\infty.$$
\ecor
\begin{proof} To simplify, assume $x=1$.
Let $b$ and $b'$ be such that $0<b<b'<1+2 \theta/\pi$. 
Then 
\begin{eqnarray*}
\EE(\sup_{s\le S}|Z_s|^b)
&=& \int_0^\infty \mathbb P[\sup_{s\le S} |Z_s|>a^{1/b}] da\;\le\; 1 + c_b \int_1^\infty a^{-b'/b} da \,<\,\infty.
\end{eqnarray*}
Let $b$ and $b'$ be such that $-1+2 \theta/\pi<b'<b<0$. Then 
\begin{eqnarray*}
\EE(\sup_{s\le S}|Z_s|^b)
&=& \int_0^\infty \mathbb P[\inf_{s\le S} |Z_s|<a^{1/b}] da\;\le\; 1 + c_b \int_1^\infty a^{-b'/b} da \,<\,\infty.
\end{eqnarray*}
\end{proof}

\bcor \label{cor24}
Let $(Z_s = (X_s,Y_s);\;s\le S)$ be distributed as
$\PP^\theta_x$. Let $f$ be an holomorphic function on an open set
containing $\mathcal Q^{\ast}$ for which $f(z)\in\mathbb R$ for
all $z\in]0,\infty[$. Assume there exist $C>0$, $b_-\in ]0,1-2
\theta/\pi[$ and $b_+\in ]0,1+2 \theta/\pi [$ such that
$$|f(z)|\le C\big(|z|^{-b_-}+|z|^{b_+}\big)\ \text{for all}\ \
z\in \mathcal{Q}^*.$$
Then setting $\phi(x,y)=\mathcal{R}\big( f(x+iy)e^{-i\theta}
\big)$, we have
$$E[\phi(iY_S)]=\cos(\theta) f(x).$$
 \ecor
\begin{proof} Recall that $(\phi(Z_{t\wedge S}))_t$ is a local
martingale (stopped at time $S$).
Using Corollary \ref{cor23}, it is a uniformly integrable
martingale. And we conclude using the martingale property.
\end{proof}

Note that the functions $f(z)=z^b$, for $b\in ]-1+2
\theta/\pi,1+2 \theta/\pi[$, $f(z)=\log(z)^\ell$ for $\ell>0$
satisfy the assumptions of Corollary \ref{cor24}.
\bcor\label{calcu} Let $(Z_s = (X_s,Y_s);\;s\le S)$ be distributed as
$\PP^\theta_x$. Then
\begin{eqnarray*}
&\bullet& E[Y_S^b] = x^b \frac{\cos(\theta)}{\cos(\theta-b
\pi/2)} \quad \hbox{ for } b\in ]-1+2\theta/\pi,1+ 2 \theta/\pi[, \\
&\bullet& \EE[\log(Y_S)]=\log(x)-\frac{\pi}{2}\ \tan(\theta),\\
&\bullet& \EE[(\log(x^{-1}Y_S))^2]= \frac{\pi^2}{4} \big(1+2
\tan^2(\theta)\big).
\end{eqnarray*}
\ecor
\begin{proof} The calculation of $\EE[Y_S^b]$ is immediate. 
Using the scaling property one only needs to do the next
calculations when $x=1$.
Now, for $x=1$ and all $\ell >0$, 
$$\EE\big[\mathcal{R}\big((\log(Y_S)+i\pi/2)^\ell
e^{-i\theta}\big)\big] = 0.$$
Applying this identity for $\ell=1$, we get the value of
$\EE[\log(Y_S)]$. For $\ell=2$, we get
$$\EE\big[\big((\log(Y_S))^2-(\pi/2)^2\big) \cos(\theta) +\pi
\log(Y_S) \sin(\theta)\big] = 0.$$
The calculation of $\EE\big[\big(\log(Y_S)\big)^2\big]$ easily follows.
\end{proof}

\begin{proof}[Proof of Proposition \ref{law:Y_S}]  Denote by $\beta'(\alpha,\beta)$ the Beta distribution of the second kind with positive parameters $\alpha$ and $\beta$. Recall it is a distribution on $\RR^+$ with density given by 
$$h(x)=\frac{x^{\alpha-1}(1+x)^{-\alpha-\beta}}{B(\alpha,\beta)},\quad x>0$$
where $B(\alpha,\beta)$ is the usual Beta function. Let $X$ be distributed as $\beta'(\alpha,\beta)$. Then for $-\alpha<t<\beta$,
$$\EE[X^t]=\frac{B(\alpha+t,\beta-t)}{B(\alpha,\beta)}.$$
Suppose that $\beta=1-\alpha$, then $\EE[X^t]=\frac{B(\alpha+t,1-\alpha-t)}{B(\alpha,1-\alpha)}.$
Using the fact that $B(z,1-z)=\Gamma(z)\Gamma(1-z)=\frac{\pi}{\sin(\pi z)}$, we get
$$\EE[X^t]=\frac{\sin(\alpha\pi)}{\sin\big((\alpha+t)\pi\big)}.$$
Thus, when $(\alpha,\beta)=\big(\frac{1}{2}-\frac{\theta}{\pi}, \frac{1}{2}+\frac{\theta}{\pi}\big)$, we have for $-\frac{1}{2}+\frac{\theta}{\pi}<t<\frac{1}{2}+\frac{\theta}{\pi}$,
$$\EE[X^t]=\frac{\sin(\pi/2-\theta)}{\sin(\pi/2-\theta+t\pi)}=\frac{\cos(\theta)}{\cos(\theta-t\pi)}.$$

Therefore, taking $x=1$, we have that for all $-\frac{1}{2}+\frac{\theta}{\pi}<t<\frac{1}{2}+\frac{\theta}{\pi}$, (replacing $b$ with $2t$ in Proposition \ref{cor24}) $\EE[X^t]=\EE[Y_S^{2t}]$. This implies Proposition \ref{law:Y_S}. \end{proof}

\subsection{Brownian motion on the quadrant with time dependent reflections}\label{secref}
Our purpose in this section and in Section \ref{deded} is to construct  a  Brownian motion in $\mathcal Q$ started at $z=(x,0)$ with $x>0$, and stopped
at its first hitting time of $(0,0)$, with time dependent reflections.

Suppose we are given a sequence of random variables
$(\Theta_n)_{n\ge 0}$
and a sequence of processes $(Z^n)_{n\ge 1}$, with $Z^n=\big(Z^n_t=(X^n_t,Y^n_t);\; t\le S_n\big)$, such that\,: 
\begin{itemize}
\item[(i)] There is $[\theta_-,\theta_+]\subset ]0,\pi/2[$ such that with probability $1$, for all $n\ge 0$, $\Theta_n\in [\theta_-,\theta_+]$.
\item[(ii)] Set $U_0=x$ and for $n\ge 1$, $U_n=Y^n_{S_n}$. Set also for $n\ge 0$,
$\mathcal{G}_n=\sigma\big((\Theta_k,Z^k);\; 1\le k\le n\big)\vee\sigma(\Theta_0).$
Then given $\mathcal G_n$, $Z^{n+1}$ is distributed as $\PP^{\Theta_n}_{U_n}$ (recall the definition of $\PP^\theta_u$ given in Section \ref{orbmhp}).
\end{itemize}
For $\theta\in ]0,\pi/2[$, set
$$v_1(\theta)=(-\tan(\theta),1)\; \hbox{ and }\;
v_2(\theta)=(1,-\tan(\theta)).$$


Set $T_0=0$ and
$T_n=\sum_{k=1}^n S_k$ for $n\ge 1$. For $n\ge 0$, set 
\begin{eqnarray*}
Z_t &=& (X^{2n+1}_{t-T_{2n}},Y^{2n+1}_{t-T_{2n}}) \quad\; \  \text{for all}\ t\in [T_{2n},T_{2n+1}[,\\
Z_t &=& (Y^{2n+2}_{t-T_{2n+1}},X^{2n+2}_{t-T_{2n+1}}) \ \ \text{for all}\ t\in [T_{2n+1},T_{2n+2}[.
\end{eqnarray*}

Thus, we have defined a process $(Z_t;\; t<T_\infty)$,
where $T_\infty=\lim_{n\to\infty}T_n$. 
For $t\ge T_\infty$, set $Z_t=(0,0)$.
Since for all $t<T_\infty$, $Z_t\ne (0,0)$, we have that $T_\infty=\sigma_0:=\inf\{t:\, Z_t=(0,0)\}$. 
However, it is not obvious that $Z$ is a continuous at $T_\infty$ (see Corollary \ref{contin}).

Note that there exists $B$, a two-dimensional Brownian motion, such that for $n\ge 0$,
$$\left\{ \begin{array}{lll}
dZ_t &=& dB_t+v_1(\Theta_{2n}) dL^1_t \quad\; \ \text{for all}\ t\in [T_{2n},T_{2n+1}[,\\
dZ_t&=& dB_t+v_2(\Theta_{2n+1}) dL^2_t \ \ \text{for all}\ t\in [T_{2n+1},T_{2n+2}[,
\end{array}\right.$$
with $L^1$ and $L^2$ respectively the local times processes of $X$ and
$Y$.
Define $(v_t;\;t<\sigma_0)$ by\,: for $n\ge 0$ 
\begin{eqnarray*}
v_t &=& v_1(\Theta_{2n}) \quad\; \  \text{for all}\ t\in [T_{2n},T_{2n+1}[,\\
v_t &=& v_2(\Theta_{2n+1}) \ \ \text{for all}\ t\in [T_{2n+1},T_{2n+2}[.
\end{eqnarray*}
Then for all $t<\sigma_0$,
\begin{equation}\label{rtrt}
Z_t=Z_0+B_t+\int_{0}^{t}v_s dL_s
\end{equation}
where $Z_0=(x,0)$ and $L=L^1+L^2$ is the accumulated local time
at $\partial\mathcal Q$ until $t$.
\subsection{The corner is reached}\label{deded}
For $a\ge 0$, set $\sigma_a:=\inf\{t:\; |Z_t|=a\}$.
Following \cite{MR792398}, we first prove that $\mathbb
P(\sigma_0\wedge\sigma_K<\infty)=1$ for all $K>x$. This is the
major difficulty we encountered here although the proof when the
angles of reflections remain constant on each boundary is quite
easy \cite[Lemma 2.1]{MR792398}.
The main idea is inspired from \cite{MR1837288}. Define for $n\ge
1$, $V_n={U_n}/{U_{n-1}}$. Then using the scaling property
(Proposition \ref{scalprop}) and the strong Markov property, we see that,  for all $n\ge 0$, given $\cG_n$, $V_{n+1}$
is distributed as
$\tilde{Y}_{\tilde{S}}$, where $\big((\tilde{X}_t,\tilde{Y}_t);\,
t\le \tilde{S}\big)$ has law $\PP^{\Theta_n}_1$.
\begin{lemma}\label{HRHR}
With probability $1$, $\sum_{n\ge 0} U_n$ is finite.
\end{lemma}
\begin{proof}
For $n\ge 1$, write
$$U_n=x\exp\bigg(\sum_{k=1}^{n}\log(V_k)\bigg).$$
We denote by $\EE_{\cG_k}$ the conditional expectation with respect to $\cG_k$.
By Corollary \ref{calcu}, for all $k\ge 1$, 
$\EE_{\cG_{k-1}}[\log(V_{k})]=-\frac{\pi}{2}\ \tan(\Theta_{k-1})$ 
and $\EE_{\cG_{k-1}}[(\log(V_{k}))^2]=\frac{\pi^2}{4} (1+2
\tan^2(\Theta_{k-1}))$.
Note now that 
$$\sum_{k=1}^{n}\log(V_k)=M_n+\sum_{k=1}^n \EE_{\cG_{k-1}}[\log(V_k)]$$
where $M_n:=\sum_{k=1}^n\big(\log(V_k)- \EE_{\cG_{k-1}}[\log(V_k)]\big)$ is a
martingale. Denote by $\langle M\rangle_n$ its quadratic
variation given by
$$\sum_{k=1}^n
\EE_{\cG_{k-1}}\big[\big(\log(V_k)-\EE_{\cG_{k-1}}[\log(V_k)]\big)^2\big]=\sum_{k=1}^{n}
\frac{\pi^2}{4} \big(1+ \tan^2(\Theta_{k-1})\big).$$
Thus $\langle M\rangle_\infty=\infty$ and so $\lim_{n\to\infty}M_n/{\langle M\rangle_n}=0$.
Since $\Theta_k\in [\theta_-,\theta_+]$, this easily implies the lemma. 
\end{proof}

A first consequence of Lemma \ref{HRHR} is
\bcor\label{contin} With probability $1$, $\lim_{t\uparrow \sigma_0} Z_t=(0,0)$.
\ecor
\begin{proof}
For $\epsilon>0$ and $n\ge 0$, set 
$$A^{\epsilon}_n=\Big\{\sup_{s\in[T_n,T_{n+1}]}|Z_s|>\epsilon\Big\}.$$

Lemma \ref{rtt} (i), with $b=1$, implies that  for all $n\ge 0$,
$$\mathbb P(A^{\epsilon}_n|\mathcal G_n)\le \epsilon^{-1} U_n \text{cotan}(\theta_-)\  + 1_{\{U_n\ge\epsilon\}}.$$
Thus, by Lemma \ref{HRHR}, $\sum_n\mathbb P(A^{\epsilon}_n|\mathcal G_n)<\infty$ a.s. for all $\epsilon>0$ and the corollary follows by applying the
conditional Borel-Cantelli lemma.

\end{proof}

Lemma \ref{HRHR} will also be used to prove
\begin{lemma}\label{ward}
For all $K>x$, $\mathbb P(\sigma_0\wedge\sigma_K<\infty)=1$.
\end{lemma}
\begin{proof}
For all $n\ge 0$ and $t\in [0,S_{n+1}]$, set 
\begin{eqnarray*}
W^{n+1}_t 
&=& \cos(\Theta_n)(X^{n+1}_t-U_{n})+\sin(\Theta_{n})Y^{n+1}_t
\end{eqnarray*}
Recall $\sigma_0=\lim_{n\to\infty}T_n$.
Define the continuous process $(W_t;\; t\le \sigma_0)$ such that
$W_0=0$ and
for $n\ge 0$ and $t\in ]T_{n},T_{n+1}]$, $W_t=W^{n+1}_{t-T_{n}} +
W_{T_{n}}$.
Then, it is straightforward to check that $(W_t;\; t\le
\sigma_0)$ is a Brownian motion stopped at $\sigma_0$.
Since for all $n\ge 0$, $U_{n}\ge 0$ and $\Theta_n\in ]0,\pi/2[$, we get that on the event
$\{\sigma_K\ge T_{n+1}\}$,
$$\sup_{t\in [T_{n},T_{n+1}]} W_t \le {2} K + W_{T_{n}}.$$
Thus, on $\{\sigma_K=\infty\}$, $\sup_{t\le \sigma_0} W_t \le {2} K + \sup_{n\ge 0}
W_{T_n}$.
Now for all $n\ge 0$, $W^{n+1}_{S_{n+1}}=\sin(\Theta_{n}) U_{n+1} -
\cos(\Theta_{n}) U_{n}\le U_{n+1}$. Note that for all $n\ge 0$, 
$$W_{T_{n+1}}-W_{T_{n}}=W^{n+1}_{S_{n+1}}.$$
This implies that on the event
$\{\sigma_K=\infty\}$, $\sup_{t\le \sigma_0} W_t \le {2} K +
\sum_{n\ge 0} U_n$, which is a.s. finite using Lemma \ref{HRHR}.
This shows that a.s. $\{\sigma_K=\infty\}\subset
\{\sigma_0<\infty\}$ and finishes the proof.
\end{proof}

And following \cite{MR792398}, we prove
\begin{theorem}\label{ww}
With probability $1$, we have $\sigma_0<\infty.$
\end{theorem}
\begin{proof} Set $b=\frac{4{\theta_-}}{\pi}$.
Let $\phi(x,y)=\mathcal{R}\big((x+iy)^{b}e^{-i{\theta_-}}\big)$, then $\phi$ is harmonic on some open set $U$ containing $\mathcal{Q}^*$ and
\begin{eqnarray*}
\nabla \phi(x,0)&=&bx^{b-1}(cos(\theta_-),\sin(\theta_-)),\\
\nabla \phi(0,y)&=&by^{b-1}(sin(\theta_-),\cos(\theta_-)).
\end{eqnarray*}
Thus for all $t<\sigma_0$ such that $Z_t\in\partial\mathcal Q$, we
have $v_t.\nabla\phi(Z_t)\le 0$. It follows from (\ref{rtrt}) and
It\^o's formula that for all $0<\epsilon<x<K$ and $t\ge 0$,
$$\EE[\phi(Z_{t\wedge \sigma_{\epsilon}\wedge \sigma_{K})}]\le
\phi(x,0).$$
Letting $t\rightarrow\infty$ and using dominated convergence, we
deduce
$$\EE[\phi(Z_{\sigma_{\epsilon}\wedge \sigma_{K})}]\le
\phi(x,0).$$
Obviously $\phi(z)\ge\cos(\theta_-) |z|^b$ for all
$z\in\mathcal{Q}$.
Setting $p_{\epsilon,K}=\mathbb P(\sigma_{\epsilon}<\sigma_{K})$,
we get
$$\cos(\theta_-)\big(\epsilon^b p_{\epsilon,K}+
K^b(1-p_{\epsilon,K})\big)\le x^b.$$
From this, we deduce 
$$p_{\epsilon,K}\ge
\frac{(K^b-x^b/\cos(\theta_-))}{K^b-\epsilon^b}.$$
As in \cite{MR792398}, since $\sigma_0\wedge\sigma_K<\infty$,
$\lim_{\epsilon\to 0}p_{\epsilon,K}=\PP(\sigma_0<\sigma_K)$, this yields
\begin{equation}\label{tgt}
\mathbb P(\sigma_0<\sigma_K)\ge 1-\frac{x^b}{K^b\cos(\theta_-)}.
\end{equation}
Letting $K\rightarrow\infty$, it comes that $\mathbb
P(\sigma_0<\infty)=1$.
\end{proof}

\begin{remark}\label{incl}
Using the inclusion $\{\sup_{t<\sigma_0}|Z_t|>\epsilon\}\subset
\{\sigma_{\epsilon}<\sigma_0\}$ and (\ref{tgt}), we deduce that
for all $\epsilon>0$,
\begin{equation}\label{gfgf}
\lim_{x\rightarrow0+}\mathbb P(\sup_{t<\sigma_0}|Z_t|>\epsilon)=0.\end{equation}
This fact will be used in Section 3.
\end{remark}

\subsection{The local time process}\label{loctip}
 Following Williams \cite{MR779455},
we prove in this section that 
\begin{theorem}\label{hjh}
With probability $1$, $L_{\sigma_0}:=\lim_{t\uparrow\sigma_0}L_t$ is finite.
\end{theorem}
\begin{proof}
In what follows, we refer to the proof of Theorem 1 in
\cite{MR779455} for more details. Let $\tilde{\theta}\in ]0,\theta_-\wedge \pi/4[$
and set $\tilde{b}=\frac{4\tilde{\theta}}{\pi}$.
Le $\tilde{\phi}$ be defined as the function $\phi$ in the proof of
Theorem \ref{ww}, with the parameters $(b,\theta_-)$ replaced by $(\tilde{b},\tilde{\theta})$.
Then there exists $c>0$ such that for all $t$ for which
$Z_t\in\partial\mathcal{Q}$, we have $v_t.\nabla\tilde{\phi}(Z_t)\le -c
|Z_t|^{\tilde{b}-1}$.
For each $\gamma>0$, define $f_{\gamma}=e^{-\gamma\tilde{\phi}}$. Then
$f_{\gamma}$ is twice continuously differentiable in
$\mathcal{Q}^*$ and
$$\Delta f_{\gamma}(z)=\gamma^2 f_{\gamma}(z) (\tilde{b} |z|^{\tilde{b}-1})^2
\hbox{ for } z\in \mathcal{Q}^*.$$
Moreover for all $t$ such that $Z_t\in\partial\mathcal{Q}$,
\begin{eqnarray*}
v_t.\nabla f_\gamma(Z_t) &=& -\gamma f_\gamma(Z_t)
\big(v_t.\nabla\tilde{\phi}(Z_t)\big).
\end{eqnarray*}
For $t<\sigma_0$, set 
$$A_t=-\gamma
\int_{0}^{t}(v_s.\nabla\tilde{\phi}(Z_s))dL_s+\frac{\gamma^2}{2}\int_{0}^{t}(\tilde{b}|Z_s|^{\tilde{b}-1})^2ds.$$
and $A_{\sigma_0}=\lim_{t\uparrow\sigma_0}A_t$. Then
\begin{eqnarray*}A_{\sigma_0}
&\ge& c \gamma \int_{0}^{\sigma_0}|Z_s|^{\tilde{b}-1}dL_s+
\frac{\gamma^2}{2}\int_{0}^{\sigma_0}(\tilde{b}|Z_s|^{\tilde{b}-1})^2ds 
\;\ge\; c \gamma \int_{0}^{\sigma_0}|Z_s|^{\tilde{b}-1}dL_s.
\end{eqnarray*}
It\^o's formula implies that for $t<\sigma_0$, 
$$f_\gamma(Z_{t})e^{-A_t}
=f_{\gamma}(Z_0)+\int_{0}^{t}e^{-A_s}(\nabla
f_{\gamma}(Z_s).dB_s).$$
Taking the expectation, we get 
$$\EE\bigg[\exp\bigg(-c\gamma\int_{0}^{\sigma_0}|Z_s|^{\tilde{b}-1}dL_s\bigg)\bigg]\ge
f_{\gamma}(Z_0).$$
This easily implies that for all $r>0$,
$$\EE\bigg[\exp\bigg(-\gamma c
r^{\tilde{b}-1}\int_{0}^{\sigma_0}1_{\{|Z_s|\le r\}}dL_s\bigg)\bigg]\ge
f_{\gamma}(Z_0).$$
Letting $\gamma\downarrow0$, we get that a.s.
\begin{equation}\label{mi}
\int_{0}^{\sigma_0}1_{\{|Z_s|\le r\}}dL_s<\infty.
\end{equation}
Let $S_r=\sup\{t\ge 0 : |Z_t|>r\}$, then by the continuity of
$Z$, $S_r<\sigma_0$ and thus $L_{S_r}<\infty$. By combining this
with (\ref{mi}), we get $L_{\sigma_0}<\infty$.
\end{proof}

\subsection{On the integrability of $L_{{\sigma_0}}$}\label{intloctip}
In this section, Proposition \ref{yes} is proved. 
We use the notations of Section \ref{secref}.
Note that $L_{\sigma_0}=\sum_{n=1}^\infty L^n_{S_n}$, 
where $L^n$ is the local time at $0$ of $Y^n$ and where $Z^n=(X^n,Y^n)$.
Recall that for $n\ge 0$, given $\mathcal{G}_n$, the law of $Z^{n+1}$ is $\mathbb{P}^{\Theta_n}_{U_n}$, where $U_0=x$ and
$U_n=Y^n_{S_n}$ for $n\ge 1$.

Let $Z^0=(X^0_t,Y^0_t)_{t\le S^0}$ be a process of law  $\mathbb{P}^{\theta}_{x}$.
Then, if $L^0_t=L_t(Y^0)$, for all $t\ge 0$,
$Y^0_{t\wedge S^0}=B^2_{t\wedge S^0} + L^0_{t\wedge S^0}$
where $(B^2_{t\wedge S^0})_t$ is a Brownian motion stopped at time $S^0$. Thus 
$\mathbb{E}[Y^0_{t\wedge S^0}]=\mathbb{E}[L^0_{t\wedge S^0}].$
Taking the limit as $t\to\infty$ and using Corollary \ref{cor23} leads to $\mathbb{E}[L^0_{S^0}]=\mathbb{E}[Y^0_{S^0}]$.
But $\mathbb{E}[Y^0_{S^0}]=x\ \text{cotan}(\theta)$ by Corollary \ref{calcu} and this implies that
$$ \mathbb{E}[L^{n+1}_{S_{n+1}} | \mathcal{G}_n] = U_n\ \text{cotan}(\Theta_n).$$
Consequently 
$$\mathbb{E}[L_{\sigma_0}] = \sum_{n\ge 0} \mathbb{E} [U_n\ \text{cotan}(\Theta_n)].$$
If for all $n$, $U_n$ and $\Theta_n$ are independent, then
$$\mathbb{E} [U_n\ \text{cotan}(\Theta_n)]
= \mathbb{E} [\text{cotan}(\Theta_n)]  \mathbb{E} [U_n]
= \cdots = x\prod_{k=0}^n \mathbb{E} [\text{cotan}(\Theta_k)].$$

If for all $n$, $\Theta_{2n}=\theta_1\in ]0,\pi/2[$ and $\Theta_{2n+1}=\theta_2\in ]0,\pi/2[$, then setting $c_1=\text{cotan}(\theta_1)$ and
$c_2=\text{cotan}(\theta_2)$,
\begin{eqnarray}
\mathbb{E}[L_{\sigma_0}] &=& x( c_1 + c_1c_2 + c_1^2c_2 + c_1^2 c_2^2+\cdots)\nonumber\\
&=& x c_1\big(1+c_2+c_1c_2+c_1c_2^2+\cdots\big)\nonumber\\
&=& x c_1\big((1+c_2)+(1+c_2)c_1c_2+\cdots\big)\nonumber\
\end{eqnarray}
which is finite if and only if $c_1c_2 <1$. In this case, we have
$\mathbb{E}[L_{\sigma_0}]=\frac{xc_1(1+c_2)}{1-c_1c_2}$
and Proposition \ref{yes} is proved.
Note that, if $\theta_1=\theta_2=\theta$, then $\mathbb{E}[L_{\sigma_0}]< \infty$ if and only if $\theta\in ]\pi/4,\pi/2[$  and in this case $\mathbb{E}[L_{\sigma_0}]=\frac{x}{\tan(\theta)-1}$.


\section{Proof of Theorem \ref{propunic}}\label{prfthm12}
Theorem \ref{propunic} (i) is proved in Section \ref{prfthm12(1)}. 
For the construction of a solution, we use the Freidlin-Sheu formula for WBM (see Theorem
\ref{nh} below).
The uniqueness in law of the solutions of $(\hbox{ISDE})$ follows from the fact that WBM is the
unique solution of a martingale problem.

Theorem \ref{propunic} (ii) is proved in Section \ref{propunic(2)}.
To prove pathwise uniqueness for $(\hbox{ISDE})$ when $N=2$, we proceed as in \cite{MR3055262} using the local times techniques
introduced in \cite{MR770393,MR658680}.
The fact that the solution of $(\hbox{ISDE})$ is not strong when $N\ge 3$ is a consequence of a theorem by Tsirelson (see
Theorem \ref{Ts} below).

We prove Theorem \ref{propunic} only for $x=0$, the case
$x\ne 0$ following easily. 
\subsection{Proof of Theorem \ref{propunic} (i)} \label{prfthm12(1)}
Let us recall the Freidlin-Sheu formula  (see \cite{MR1743769} and
also \cite[Theorem 3]{MR2835247}).
\begin{theorem}\label{nh} \cite{MR1743769} Let $(X_t)_{t\ge 0}$ be a Walsh's
Brownian motion on $G$ and
$B^X_t=|X_t|-|X_0|-L_t(|X|)$. Then $B^X$ is a Brownian motion and
for all $f\in C^2_b(G^{\ast})$, we have
$$f(X_t)=f(X_0)+\int_0^tf'(X_s) dB^X_s+\frac{1}{2}\int_0^t f''(X_s)
ds+f'(0) L_t(|X|).$$
We call $B^X$ the Brownian motion associated to $X$.
\end{theorem}
Remark that in this formula the local martingale part of $f(X_t)$
is always a stochastic integral with respect to $B^X$. This is an
expected fact since $B^X$ has the martingale representation
property for $(\mathcal F^X_t)_t$ (\cite[Theorem 4.1]{MR1022917}).
This martingale representation property will be used 
to prove the uniqueness in law of the solutions to $(\hbox{ISDE})$.

\subsubsection{Construction of a solution of $(\hbox{ISDE})$}\label{descr}
Let $X$ be a WBM with $X_0=0$ and let $B^X$
be the Brownian motion associated to $X$. Take an $N$-dimensional
Brownian motion $V=(V^{1},\cdots,V^{N})$ independent of $X$. Let
$(\mathcal F_t)$ denote the filtration generated by $X$ and $V$.
For $i\in[1,N]$, define
$$W^{i}_t=\int_{0}^{t} 1_{\{X_s\in E_{i}\}}dB^X_s + \int_{0}^{t}
1_{\{X_s\notin E_{i}\}}dV^{i}_s.$$
Then $W:=(W^1,\cdots,W^N)$ is an $N$-dimensional $(\mathcal
F_t)$-Brownian motion by L\'evy's theorem and
$$B^X_t=\sum_{i=1}^{N}\int_{0}^{t}1_{\{X_s\in E_i\}}dW^i_s.$$
Then, using Theorem \ref{nh}, $(X,W)$ solves $(\hbox{ISDE})$. Denote by $\mu$ the law of $(X,W)$.

\subsubsection{Uniqueness in law}\label{descu}
To prove the uniqueness in law, we apply two lemmas. The first lemma states that the WBM
is the unique solution of a martingale problem.
The second lemma gives conditions that ensure that a Walsh's
Brownian motion is independent of a given family of Brownian motions.
\begin{lemma}\label{martprop}
Let $(\mathcal{F}_t)$ be a filtration and $X$ be a $G$-valued
$(\mathcal{F}_t)$-adapted continuous process such that for
all $f\in\mathcal{D}$,
\beq M^f_t:= f(X_t)-f(x)-\frac{1}{2}\int_0^t f''(X_s) ds.
\label{MPeq}\eeq
is a martingale with respect to $(\mathcal{F}_t)$. Then $X$ is an
$(\mathcal F_t)$-WBM.
\end{lemma}
\begin{proof}
We exactly follow the proof of \cite[Theorem 3.2]{MR1022917}
and only check that with our conventions $f'(0)=f''(0)=0$
for $f\in\mathcal D$, we avoid all trivial solutions to the
previous martingale problem (with the hypothesis of Theorem 3.2
of \cite{MR1022917}, the trivial process $X_t=0$ is a possible
solution of the martingale problem (3.3) in \cite{MR1022917}).
For $i\in [1,N]$, set $q_i=1-p_i$ and let $f_i$ and $g_i$ be
defined by
\begin{eqnarray*}
f_i(x) &=& q_i|x| 1_{\{x\in E_i\}} - p_i|x| 1_{\{x\not\in
E_i\}},\\
g_i(x)&=& \big(f_i(x)\big)^2= q_i^2|x|^2 1_{\{x\in E_i\}} +
p_i^2|x|^2 1_{\{x\not\in E_i\}}.
\end{eqnarray*}
Then $f_i$ and $g_i$ are $C^2$ on $G^*$.
We have $f'_i(x)=q_i$ for $x\in E_i^*$, $f'_i(x)=-p_i$ for
$x\not\in E_i$ and $f'_i(0)=0$.
Moreover, $f''_i(x)=0$ for $x\in G$.
We also have $g'_i(x)=2 q_i^2|x|$ for $x\in E_i^*$,
$g'_i(x)=2p_i^2|x|$ for $x\not\in E_i$ and $g'_i(0)=0$.
Moreover, $g''_i(x)=2 q_i^2$ for $x\in E_i^*$, $g''_i(x)=2 p_i^2$
for $x\not\in E_i$ and $g''_i(0)=2 p_i q_i$.
Set $Y^i_t:=f_i(Z_{t})$. 
Although $f_i$ is not bounded, by a localization argument, we
have that $Y^i_t$ is a local martingale. We also have that $(Y^i_t)^2 - \frac{1}{2}\int_0^{t}
g''_i(Z_s) ds$ is a local martingale.
Thus 
$$\langle Y^i\rangle_t=
\int_0^{t} \big(q_i^2 1_{\{Z_s\in E_i^*\}} 
+ p_i^2 1_{\{Z_s\not\in E_i\}} 
+ p_i q_i 1_{\{Z_s = 0\}}\big) ds.$$
Set $$U^i_t= \int_0^t \big(q_i^{-1} 1_{\{Y^i_s>0\}} + p_i^{-1}
1_{\{Y^i_s< 0\}} +\big(p_iq_i\big)^{-1/2}1_{\{Y^i_s= 0\}}\big)
dY^i_s.$$
Then $U^i_t$ is a local martingale with $\langle U^i\rangle_t=t$; that is $U^i_t$ is a Brownian motion.
Let $\phi(y)=q_i 1_{\{y>0\}} + p_i 1_{\{y< 0\}}
+\sqrt{p_iq_i}1_{\{y= 0\}}$.
Then $Y^i$ is a solution of the stochastic differential
equation
$$Y^i_t=Y^i_0+\int_0^t \phi(Y^i_s)dU^i_s.$$
As in \cite{MR1022917}, the solution of this SDE is pathwise
unique and following the end of the proof of \cite[Theorem 3.2]{MR1022917}, we arrive at
$$\EE[f(Z_t)|\mathcal F_s]=P_{t-s}f(Z_s)$$
for all $s\le t$ and $f\in C_0(G)$, and where $P_t$ is the semigroup of the Walsh's Brownian
motion.
\end{proof}

\begin{lemma}\label{mawl0}
Let $(\mathcal G_t)$ be a filtration, $X$ be a $(\mathcal G_t)$-WBM and 
$B=(B^1,\cdots,B^d)$ be a $(\mathcal G_t)$-Brownian motion in $\mathbb R^d$, with $d\ge 1$. 
Denote by $B^X$ the Brownian motion associated to $X$.
Then $B^X$ and $B$ are independent if and only if $X$ and $B$ are independent.
\end{lemma}
\begin{proof}
Clearly, if $X$ and $B$ are independent, then $B^X$ and $B$ are independent. Let us prove the converse.
Let $U$ be a bounded $\sigma(B)$-measurable random variable. Then$$U=\EE[U]+\sum_{i=1}^d
\int_{0}^{\infty}H^i_s dB^i_s$$
with $H^i$ predictable for the filtration $\mathcal F^{B}_\cdot$
and $E[\int_{0}^{\infty}(H^i_s)^2 ds]<\infty$. Let $U'$ be a
bounded $\sigma(X)$-measurable random variable. Since $B^X$ has
the martingale representation property for $\mathcal F^X_\cdot$ \cite[Theorem 4.1]{MR1022917}, we deduce that
$$U'=\EE[U']+\int_{0}^{\infty}H_s dB^X_s$$
with $H$ predictable for $\mathcal F^X_\cdot$ and
$\EE[\int_{0}^{\infty}(H_s)^2 ds]<\infty$. Then $H$ and
$(H^i)_{1\le i\le d}$ are also predictable for $(\mathcal G_t)$.
It is also easy to check that $B^X$ is a $(\mathcal
G_t)$-Brownian motion. Now
\begin{eqnarray*}
\EE[UU']&=&
\EE[U]\EE[U']+\EE\left[\sum_{i=1}^{d}\int_{0}^{\infty}H^i_s
dB^i_s\int_{0}^{\infty}H_s dB^X_s\right]\\
&=& \EE[U]\EE[U']+\sum_{i=1}^{d} \EE\left[\int_{0}^{\infty}H^i_s
H_s d\langle B^i,B^X\rangle_s\right]\\
&=& \EE[U]\EE[U'].
\end{eqnarray*}
\end{proof}

Let $(X,W)$ be a solution of $(\hbox{ISDE})$ defined on a filtered probability space $(\Omega,(\mathcal F_t),\PP)$ and such that $X_0=0$.
Without loss of generality, we can assume that $\mathcal F_t=\mathcal F^X_t\vee \mathcal F^W_t$.
For all $f\in \mathcal{D}$,
$\sum_{i=1}^N \int_{0}^{t}f'(X_s)1_{\{X_s\in E_i\}}dW^i_s$ is a
martingale and therefore $X$ is a solution to the martingale problem of Lemma
\ref{martprop}. Thus $X$ is a WBM.
Let $B$ be a Brownian motion independent of $(X,W)$, denote by
$B^X$ the Brownian motion associated to $X$ and set $\mathcal
G_t={\mathcal F_t}\vee \mathcal F^B_t$.
Note that $B^X$ is a $(\cG_t)$-Brownian motion.
For $i\in[1,N]$, define
$$V^{i}_t=\int_{0}^{t} 1_{\{X_s\in E_{i}\}}dB_s + \int_{0}^{t}
1_{\{X_s\notin E_{i}\}}dW^{i}_s.$$
Then $V:=(V^{1},\cdots,V^{N})$ is an $N$-dimensional $(\cG_t)$-Brownian
motion independent of $B^X$. By the
previous lemma $V$ is also independent of $X$.
It is easy to check that for all $i\in[1,N]$,
$$W^{i}_t=\int_{0}^{t} 1_{\{X_s\in E_{i}\}}dB^X_s + \int_{0}^{t}
1_{\{X_s\notin E_{i}\}}dV^{i}_s.$$
This proves that the law of $(X,W)$ is $\mu$. 

\subsection{Proof of Theorem \ref{propunic} (ii)} \label{propunic(2)}
\subsubsection{The case $N=2$}\label{doudou} 
To prove that the solution is a strong one, it suffices to prove that pathwise
uniqueness holds for $(\hbox{ISDE})$.
Fix $p\in]0,1[$, and set $\beta=\frac{1-p}{p}$. 
\begin{lemma}\label{gdd}
Let $X$ be a continuous process, $B^+$ and $B^-$ be two independent Brownian motions.  Set $Y_t=\beta
X_t 1_{\{X\geq 0\}}+X_t 1_{\{X_t<0\}}$.
Then $(X,B^+,B^-)$ is a solution to $(\hbox{ISDE})$ or equivalently of
\begin{equation}\label{ess}
dX_t=1_{\{X_t>0\}}dB^+_t+1_{\{X_t\leq0\}}dB^-_t+(2p-1)dL_t(X)
\end{equation}
 if and only if $(Y,B^+,B^-)$ is a solution of the following SDE
\begin{equation}\label{hqq}
dY_t=\beta 1_{\{Y_t>0\}}dB^+_t+1_{\{Y_t\leq0\}}dB^-_t.
\end{equation}
\end{lemma}
\begin{proof}
Suppose $(X,B^+,B^-)$ solves \eqref{ess}. Set $B_t=\int_0^t 1_{\{X_s>0\}}dB^+_s+\int_0^t1_{\{X_s\leq0\}}dB^-_s$.
Then $B_t$ is a Brownian motion and $(X,B)$ is a solution of the SDE $dX_t=dB_t+(2p-1)dL_t(X)$.  
It is well known (see for example Section 5.2 in the survey \cite{MR2280299}) that $(Y,B)$ solves
$$dY_t=\beta 1_{\{Y_t>0\}}dB_t+1_{\{Y_t\leq0\}}dB_t$$
and thus that $(Y,B^+,B^-)$ solves \eqref{hqq}. The converse can be proved in the same way.
\end{proof}

\begin{proposition}
Pathwise uniqueness holds for $(\hbox{ISDE})$.
\end{proposition}
\begin{proof}
Lemma \ref{gdd} implies that the proposition holds if pathwise uniqueness holds for \eqref{hqq}.
Let $(Y,B^+,B^-)$ and $(Y',B^+,B^-)$ be two solutions of \eqref{hqq} with $Y_0=Y'_0=0$.
Set $\text{sgn}(y)=\one_{\{y>0\}}-\one_{\{y<0\}}$. We shall use the same techniques as in
\cite{MR3055262} (see also \cite{MR770393,MR658680})
and first prove that a.s.
\begin{equation}\label{ddd}
\int_{]0,+\infty]}^{} L^a_t(Y-Y')\frac{da}{a}<\infty.
\end{equation}
By the occupation times formula
$$\int_{]0,+\infty]}^{}L^a_t(Y-Y')\frac{da}{a}=\int_{0}^{t}1_{\{Y_s-Y'_s>0\}}\frac{d\langle
Y-Y'\rangle_s}{Y_s-Y'_s}.$$
It is easily verified that
$$d\langle Y-Y'\rangle_s\le C
\big|\text{sgn}(Y_s)-\text{sgn}(Y'_s)\big|ds$$
where $C=(1+\beta^2)/2$. 
Let $(f_n)_n\subset C^1(\RR)$
such that $f_n\rightarrow\text{sgn}$ pointwise and $(f_n)_n$ is
uniformly bounded in total variation. By Fatou's lemma, we get
\begin{eqnarray}
\int_{]0,+\infty]}^{}L^a_t(Y-Y')\frac{da}{a}&\le&
C\liminf_n\int_{0}^{t}1_{\{Y_s-Y'_s>0\}}\frac{|f_n(Y_s)-f_n(Y'_s)|}{Y_s-Y'_s}ds\nonumber\\
&\le& C\liminf_n\int_{0}^{t}1_{\{Y_s-Y'_s>0\}}\bigg|\int_{0}^{1}
f'_n(Z^u_s)du\bigg|ds\nonumber\
\end{eqnarray}
where 
$$Z^u_s=(1-u)Y_s+uY'_s.$$
It is easy to check the existence of a constant $A>0$ such that
for all $s\ge0$ and $u\in[0,1]$, $\frac{d}{du}\langle Z^u \rangle_s\ge A^{-1}.$
Hence, setting $C'=A\times C$, we have
\begin{eqnarray}
\int_{]0,+\infty]}^{}L^a_t(Y-Y')\frac{da}{a}&\le&
C'\liminf_n\int_{0}^{1} \int_{0}^{t}\big|f'_n(Z^u_s)\big|d\langle
Z^u \rangle_sdu\nonumber\\
&\le& C'\liminf_n\int_{0}^{1}\int_{\RR} \big|f'_n(a)\big| L^a_t(Z^u)dadu.\nonumber\
\end{eqnarray}
Now taking the expectation and using Fatou's lemma, we get
$$\EE\bigg[\int_{]0,+\infty]}^{}L^a_t(Y-Y')\frac{da}{a}\bigg]\le
C'\liminf_n\int_{\RR} \big|f'_n(a)\big|da 
\sup_{a\in\RR, u\in[0,1]} \EE\big[L^a_t(Z^u)\big].$$
It remains to prove that $\sup_{a\in\RR,
u\in[0,1]}\EE\big[L^a_t(Z^u)\big]<\infty$. By Tanaka's formula, we
have
\begin{eqnarray*}
\EE\big[L^a_t(Z^u)\big]&=&
\EE\big[\big|Z^u_t-a\big|\big]-\EE\big[\big|Z^u_0-a\big|\big]-\EE\bigg[\int_{0}^{t}\text{sgn}(Z^u_s-a)dZ^u_s\bigg]\\
&\le& E[\big|Z^u_t-Z^u_0\big|].
\end{eqnarray*}
It is easy to check that the right-hand side is uniformly bounded
with respect to $(a,u)$ which permits to deduce (\ref{ddd}).
Consequently, since $\lim_{a\downarrow 0} L^a(Y-Y')=L^0(Y-Y')$, 
\eqref{ddd} implies that $L^0_t(Y-Y')=0$ and thus by Tanaka's formula, 
$|Y-Y'|$ is a local martingale which is also a nonnegative supermartingale, 
with $|Y_0-Y'_0|=0$ and finally $Y$ and $Y'$ are indistinguishable.
\end{proof}

\subsubsection{The case $N\ge 3$}\label{N>=3}
Let $(X,W)$ be a solution to $(\hbox{ISDE})$. Then $X$ is an $(\cF_t)$-WBM, where $\cF_t=\cF^X_t\vee\cF^W_t$. If
$(X,W)$ is a strong solution, we thus have that $X$ is an $(\cF^W_t)$-WBM, which is impossible when $N\ge
3$ because of the following Tsirelson's theorem:
\begin{theorem}\label{Ts}\cite{MR1487755}
There does not exist any $(\mathcal G_t)_t$-Walsh's Brownian
motion on a star graph with three or more rays with $(\mathcal G_t)_t$ a Brownian filtration.
\end{theorem}

\section{Proof of Theorem \ref{thmphi}}\label{S3}
In this section, we prove the assertions (i) and (ii) of Theorem \ref{thmphi}.
We first construct a coalescing SFM solution of $(\hbox{ISDE})$.
To construct this SFM, we will use the following
\begin{theorem}\label{610}\cite{MR2060298}
Let $(P^{(n)},n\geq 1)$ be a consistent family of Feller semigroups
acting respectively on $C_0(M^n)$ where $M$ is a locally compact metric space such that
\begin{equation}\label{hila}
P^{(2)}_tf^{\otimes 2}(x,x)=P^{(1)}_tf^2(x)\ \textrm{for all}\ f\in
C_0(M),\ x\in M,\ t\geq0.
\end{equation}
Then there is a (unique in law) SFM $\varphi=(\varphi_{s,t})_{s\le t}$ defined 
on some probability space $(\Omega,\mathcal A,\mathbb P)$
such that
$$P^{(n)}_tf(x)=\EE[f(\varphi_{0,t}(x_1),\cdots,\varphi_{0,t}(x_n))]$$
for all $n\ge 1$, $t\ge 0$, $f\in C_0(M^n)$ and $x\in M^n$. 
\end{theorem}
To apply this theorem, we construct a consistent family of
$n$-point motions (i.e. the Markov process associated to $P^{(n)}$) 
up to their first coalescing times in Section \ref{qqq0}.
After associating to the two-point motion an obliquely reflected Brownian motion in $\mathcal{Q}$ in Section \ref{orbm2pt},
we prove the coalescing property in Section \ref{label} and the Feller property in Section \ref{constphi}. 
It is then possible to apply Theorem \ref{610} and as a result we get a flow $\varphi$. In Section \ref{constphi}, we show that $\varphi$ solves $(\hbox{ISDE})$.
Finally, we prove in Section \ref{eqq} that $\varphi$ is the unique SFM solving $(\hbox{ISDE})$.  

\smallskip
Note finally that in the case of Le Jan and Raimond \cite{MR000}, all the angles of reflection of the obliquely reflected
Brownian motion associated to the two-point motion are equal to $\pi/4$. This simplifies greatly the study of Section
\ref{tdrbm}.

\subsection{Construction of the $n$-point motion up to the first coalescing time}\label{qqq0}

Fix $x=(x_1,\cdots,x_n)\in G^n\setminus \Delta_n$, 
where $\Delta_n:=\{x\in G^n:\, \exists i\ne j,\, x_i=x_j\}$. 
Let $(Y,W)$ be a solution of  $(\hbox{ISDE})$, with $Y_0=0$.
For $t\ge 0$ and $j\le n$, set 
$$X^{j,0}_t\;= \left\{\begin{array}{ll} Y_t & \hbox{if $x_j=0$},\\
e_i(|x_j|+W^i_t) &\hbox{if $x_j\ne 0$ and $i$ is such that $x_j\in E_i^\ast$}.\end{array}\right.$$
Set $\tau_0=0$, $r=\inf\{|x_j|:\, 1\le j\le n\}$ and
 $$\tau_1\;= \left\{\begin{array}{ll}
\inf\{t\ge 0 :\, \exists j \hbox{ such that } X^{j,0}_t=0\},& \hbox{ if $r>0$},\\
\inf\{t\ge 0 :\, \exists j\ne j_0 \hbox{ such that} X^{j,0}_t=0\}, & \hbox{ if $r=0$ and where $j_0$ is such that $x_{j_0}=0$.}\end{array}\right.$$
For $t\le \tau_1$, set $X^j_t=X^{j,0}_t$ and let $X^{(n)}_t=(X^{1}_t,\cdots,X^{n}_t)$. Note that a.s. $X^{(n)}_{\tau_1}\not\in\Delta_n$.

Assume now that $(\tau_k)_{k\le \ell}$ and $(X^{(n)}_t=\big(X^{1}_t,\cdots,X^{n}_t)\big)_{t\le\tau_l}$ have been defined such that a.s. 
for all $1\le k\le \ell$, $X^{(n)}_{\tau_k}\not\in \Delta_n$ and there is a unique integer $j_k$ such that ${X}^{j_k}_{\tau_k}=0$.
Now introduce an independent solution $(Y^\ell,W^\ell)$ of $(\hbox{ISDE})$, with
$Y^\ell_0=0$, and define $(X^{(n)}_t)_{t\in[\tau_\ell,\tau_{\ell+1}]}$ by analogy
with the construction of $(X^{(n)}_t)_{t\in[0,\tau_1]}$ by
replacing $x$ with $X^{(n)}_{\tau_\ell}$.
Thus, we have defined $X^{(n)}_t$ for all $t<\tau_{\infty}$, where $\tau_{\infty}:=\lim_{n\rightarrow\infty}\tau_n$. 
Moreover, by construction, $(\tau_n)_{n\ge 0}$  is an increasing sequence of
stopping times with respect to the filtration associated to $X^{(n)}$.

We denote by $\PP^{(n),0}_x$ the law of $(X^{(n)}_t)_{t<\tau_{\infty}}$. 
Notice that for all $i$ and all $\ell$, $(X^i_{t\wedge \tau_\ell})$ is a WBM stopped at time $\tau_\ell$. 
Thus a.s. on the event  $\{\tau_\infty<\infty\}$, 
$X^{(n)}_{\tau_\infty}:=\lim_{t\uparrow\tau_\infty} X^{(n)}_t$ exists (this is a well known result for the standard Brownian motion, see \cite[Proposition 3.3]{MRCVCVCV}, which can easily be extended to WBM). Note also that there exist $i\ne j$ such that, a.s. on the event $\{\tau_\infty<\infty\}$, 
 $X^i_{\tau_\ell}=X^j_{\tau_{\ell+1}}=0$ for infinitely many $\ell$'s and thus that
$\lim_{t\uparrow\tau_\infty} X^i_t=\lim_{t\uparrow\tau_\infty} X^j_t=0$.
Therefore, $X^{(n)}_{\tau_\infty}\in\Delta_n$ a.s.
By construction, this implies that 
\begin{equation}\label{asa}
\tau_\infty=T_{\Delta_n}:=\inf\{t\ge 0 : X^{(n)}_t\in\Delta_n\}.
\end{equation}
We will prove in Section \ref{label} that $\tau_\infty < \infty$ a.s.

\subsection{An obliquely reflected Brownian motion associated to the $2$-point motion}\label{orbm2pt}
Fix $x\in G^{\ast}$. Let $(X,Y)$ be the process with law $\PP^{(2),0}_{(x,0)}$ constructed in Section \ref{qqq0}. Then for all $n\ge 0$,
\begin{eqnarray}
\tau_{2n+1}&=&\inf\{t\ge \tau_{2n} : X_t=0\},\nonumber\\
\tau_{2n+2}&=&\inf\{t\ge \tau_{2n+1} : Y_t=0\}.\nonumber\
\end{eqnarray} 
Letting, for all $n\ge 0$,  $i_{2n}$ and $i_{2n+1}$  in $\{1,\dots,N\}$ be such
that $X_{\tau_{2n}}\in E_{i_{2n}}$ and $Y_{\tau_{2n+1}}\in
E_{i_{2n+1}}$, we have $X_t\in E_{i_{2n}}^\ast$ for all  $t\in  [\tau_{2n},\tau_{2n+1}[$ and 
$Y_t\in E_{i_{2n+1}}^\ast$ for all  $t\in [\tau_{2n+1},\tau_{2n+2}[$.

Define, for $i\in [1,N]$, $f^i:G\to\RR$ by 
$$f^i(x)=-|x|\ \text{if}\ x\in E_i\quad \text{and}\quad
f^i(x)=|x|\ \text{if not.}$$
Define now $(U_t,V_t)_{t <\tau_\infty}$ such that for $n\ge 0$
$$(U_t,V_t) = \left\{\begin{array}{ll}(|X_t|,f^{i_{2n}}(Y_t)) &\hbox{ for } t\in [\tau_{2n},\tau_{2n+1}[;\\ 
(f^{i_{2n+1}}(X_t),|Y_t|) &\hbox{ for } t\in [\tau_{2n+1},\tau_{2n+2}[.\end{array}\right.$$
Remark that $(U_t,V_t)_{t<\tau_{\infty}}$ is a continuous process with values
in $\{(u,v)\in\mathbb R^2:\; u+v>0\}$ and such that for all
$n\ge 0$, $U_{\tau_{2n}}>0$, $V_{\tau_{2n}}=0$,
$U_{\tau_{2n+1}}=0$ and $V_{\tau_{2n+1}}>0$. Note that the excursions of this process outside of
$\mathcal{Q}$ occur on straight lines parallel to $\{y=-x\}$.

Let, for $n\ge 0$,
$$\Theta_n=\arctan\bigg(\frac{p_{i_{n}}}{1-p_{i_{n}}}\bigg).$$
Define for $t<\tau_{\infty}$, $A(t)$ the amount of time where $X$ and $Y$ do not both belong to the same ray before time $t$. Note that  $A(t)=\int_{0}^{t}1_{\{(U_s,V_s)\in\mathcal{Q}\}} ds.$

Set $\gamma(t)=\inf\{s\ge 0 : A(s)>t\}.$
Set for $n\ge 0$, $T_n=A(\tau_n)$ and $S_{n+1}=T_{n+1}-T_n$.
Define  for $t< T_\infty:=\lim_{n\to \infty} T_n$,
$$(U^r_t,V^r_t)=(U_{\gamma(t)},V_{\gamma(t)})$$
and for $t\ge T_{\infty}$,  $(U^r_t,V^r_t)=(0,0)$.
Note that $T_{2n+1}=\inf\{t\ge T_{2n} : V^r_t=0\}$, $T_{2n+2}=\inf\{t\ge T_{2n+1} : U^r_t=0\}$
and that $\gamma(T_n)=\tau_n$.

\begin{lemma}\label{zzz} 
Given $\Theta_0$, the law of $(U^r_t,V^r_t)_{t\le S_1}$ is $\PP^{\Theta_0}_{|x|}$.
\end{lemma}

The proof of this lemma is given at the end of this section. 
Define the sequence of processes $(Z^n)_{n\ge 1}$ such that for $n\ge 0$,
\begin{eqnarray*}
Z^{2n+1} &=& (U^r_{t+T_{2n}},V^r_{t+T_{2n}})_{t\le S_{2n+1}},\\
Z^{2n+2} &=& (V^r_{t+T_{2n+1}},U^r_{t+T_{2n+1}})_{t\le S_{2n+2}}.
\end{eqnarray*}
Set also for $n\ge 0$, $U_{2n}=U^r_{T_{2n}}$ and $U_{2n+1}=V^r_{T_{2n+1}}$.

Applying Lemma \ref{zzz} and using the strong Markov property at the stopping times $\tau_{n}$, 
with the fact that if $(X,Y)$ is distributed as $\PP^{(2),0}_{(x,y)}$, 
then $(Y,X)$ is distributed as $\PP^{(2),0}_{(y,x)}$, one has the following

\begin{lemma}\label{zzzn} 
For all $n\ge 0$, given $\mathcal{F}_{\tau_{n}}$, the law of $Z^{n+1}$  is $\PP^{\Theta_{n}}_{U_n}$.
\end{lemma}
This lemma shows that the sequences $(\Theta_n)_{n\ge 0}$ and $(Z^n)_{n\ge 1}$ satisfy (i) and (ii) in the beginning of
Section \ref{secref} since for all $n\ge 0$,
$$\mathcal{G}_n=\sigma\big((\Theta_k,Z^k);\; 1\le k\le n\big)\vee\sigma(\Theta_0)\subset \mathcal F_{\tau_{n}}.$$ 
Thus $(U^r_t,V^r_t)_{t< T_\infty}$ is a Brownian motion in $\mathcal{Q}^*$ started from
$(|x|,0)$, with time dependent angle of reflections at the boundaries
given by $(\Theta_n)_{n\ge 0}$
and stopped when  it hits $(0,0)$, as defined in Section 2.
In particular,  $(U^r,V^r)$ is a continuous process and 
$\lim_{t\uparrow T_\infty} (U^{r}_t,V^{r}_t)=(0,0).$ 

\begin{remark} Note that $(i_n)_{n\ge 0}$ is an homogeneous Markov chain started from $i_0=1$ with transition matrix
$(P_{i,j})$ given by : for $(i, j)\in[1,N]^2$, $P_{i,j}=\frac{p_j}{\sum_{k\ne i} p_k}$.
Remark also that given $\mathcal G_n$, $Z^{n+1}$ and $i_{n+1}$ are independent and a fortiori $Z^{n+1}$ and $\Theta_{n+1}$
are also independent.
\end{remark}

\noindent{\bf Proof of Lemma \ref{zzz}.}
Set for $t\ge 0$, $(\hat{U}_t,\hat{V}_t):=(|x|+W^{i_0}_t,f^{i_0}(Y_t))$.
Then, for $t\le \tau_1$, $(\hat{U}_t,\hat{V}_t)=(U_t,V_t)$. 
Since $Y$ is a WBM started at $0$,
it is well known that $\hat{V}$ is a skew Brownian motion with parameter $1-p_{i_0}$. 
This can be seen using Freidlin-Sheu formula, which shows that
\begin{equation}\label{BMR}
\hat{V}_t = \int_0^t \big( \one_{\{\hat{V}_s>0\}} - \one_{\{\hat{V}_s\le 0\}}\big) dB^Y_s + (1-2p_{i_0})L_t(\hat{V}).
\end{equation}
Define $\hat{A}(t) = \int_0^t \one_{\{\hat{V}_s \ge 0\}} ds = \int_0^t \one_{\{Y_s\not\in E_{i_0}\}} ds$ and $\hat{\gamma}(t)=\inf\{s\ge 0 :\;\hat{A}(s)>t\}$. 
Note that $\hat{A}(t)=A(t)$ for all $t\le \tau_1$ and thus $\hat{\gamma}(t)=\gamma(t)$ for all $t\le S_1=A(\tau_1)$.
It is also well known that $\hat{V}^r_t:=\hat{V}_{\hat{\gamma}(t)}$ is a reflecting Brownian motion on $\RR_+$.
Set $M_t = \int_0^t 1_{\{\hat{V}_s> 0\}} d \hat{V}_s=\int_0^t 1_{\{Y_s\not\in E_i\}} d B^Y_s$.
Then $B^2_t:=M_{\hat{\gamma}(t)}$ is a Brownian motion.
We also have $\hat{V}_t\vee 0 = M_t + (1-p_{i_0}) L_t(\hat{V})$, which implies that $\hat{V}^r_t=B^2_t + (1-p_{i_0}) L_{\hat{\gamma}(t)}(\hat{V})$ and therefore $L_t(\hat{V}^r)=(1-p_{i_0})L_{\gamma(t)}(\hat{V})$.
Note finally that $L(\hat{V})=L(|Y|)$.

Set for $t\ge 0$, $B^1_t=\int_{0}^{\hat\gamma(t)}1_{\{\hat{V}_s>0\}} dW^{i_0}_s$.
By L\'evy's theorem $B^1$ and $B^2$ are two independent Brownian motions.
Finally, set $\hat{U}_t^r=\hat{U}_{\hat{\gamma}(t)}$. 
Then, for all $t\le S_1$, $(\hat{U}_t^r,\hat{V}_t^r)=(U_t^r,V_t^r)$.

Lemma \ref{zzz} is a direct consequence of the following

\begin{lemma}\label{tyjh} For all $t\ge 0$, 
\begin{eqnarray*}
\hat{U}^r_t &=& |x| + B^1_t - \frac{p_i}{1-p_i}L_t(\hat{V}^r)\\
\hat{V}^r_t &=&  B^2_t + L_t(\hat{V}^r).
\end{eqnarray*}
\end{lemma}
\begin{proof}
We closely follow the proof of Lemma 4.3  \cite{MR000}. Let $\eps>0$ and define the sequences of stopping times
$\sigma_k^\eps$ and $\tau_k^\eps$ such that $\tau_0^\eps=0$ and
for $k\ge 0$,
\begin{eqnarray}
\sigma_{k}^\eps &=& \inf\{t\ge \tau_{k}^\eps;\;\hat{V}_t=-\eps\},\nonumber\\
\tau_{k+1}^\eps &=& \inf\{t\ge \sigma_{k}^\eps;\; \hat{V}_t=0\}.\nonumber\
\end{eqnarray}
Note first that \eqref{BMR} implies that
$$\sum_{k\ge 0} \big(\hat{V}_{\sigma^\eps_{k}\wedge
\hat\gamma(t)}-\hat{V}_{\tau^\eps_{k}\wedge \hat\gamma(t)}\big)$$
converges in probability as $\eps\to 0$ towards $B^2_t + (1-2p_i) L_{\hat\gamma(t)}(\hat{V})$.
Since $\hat{V}_t^r=B^2_t + (1-p_i)L_{\hat\gamma(t)}(\hat{V})$, setting
$$L^{\eps,r}_t = \sum_{k\ge 0} \big(\hat{V}_{\tau^\eps_{k+1}\wedge
\hat\gamma(t)}-\hat{V}_{\sigma^\eps_{k}\wedge \hat\gamma(t)}\big),$$
$L^{\eps,r}_t$ converges towards $p_i L_{\hat\gamma(t)}(\hat{V})$ in probability as $\eps\to 0$. Now for
$t>0$,
$$\hat{U}^{r}_t = |x| + \sum_{k\ge 0} \big(\hat{U}_{\tau^\eps_{k+1}\wedge
\hat\gamma(t)}-\hat{U}_{\tau^\eps_{k}\wedge \hat\gamma(t)}\big).$$
Set for $t\ge 0$,
$$B^{\eps,1}_t=\sum_{k\ge 0} \big(W^i_{\sigma^\eps_{k}\wedge
\hat\gamma(t)}-W^i_{\tau^\eps_{k}\wedge \hat\gamma(t)}\big).$$
Note that $d(\hat{U}_s+\hat{V}_s)=\sum_{j\neq i} \one_{\{Y_s\in E_j\}} dW^j_s$
and thus when $Y_s\in E_i^*$ (i.e. $\hat{V}_s$ is negative), $\hat{U}_s+\hat{V}_s$ remains constant, and we have 
\begin{eqnarray*}
\hat{U}_t^r
&=& |x| + \sum_{k\ge 0} \big(\hat{U}_{\tau^\eps_{k+1}\wedge
\hat\gamma(t)}-\hat{U}_{\sigma^\eps_{k}\wedge \hat\gamma(t)}\big) + \sum_{k\ge 0} \big(\hat{U}_{\sigma^\eps_k\wedge
\hat\gamma(t)}-\hat{U}_{\tau^\eps_{k}\wedge \hat\gamma(t)}\big) \\
&=& |x| - L^{\eps,r}_t + B^{\eps,1}_t.
\end{eqnarray*}
Since $B^{\eps,1}_t$ converges in probability towards $B^1_t$, we get
$$\hat{U}_t^r = |x| + B^1_t - p_{i_0} L_{\hat\gamma(t)}(\hat{V}).$$
And we conclude using that $L_t(\hat{V}^r)=(1-p_{i_0}) L_{\hat\gamma(t)}(\hat{V})$.
\end{proof}

\subsection{Coalescing property}\label{label}
Our purpose in this section is to prove that the consistent family defined in Section \ref{qqq0} has the coalescing property, that is

\begin{proposition}\label{coal}
With probability $1$, $\tau_{\infty}<\infty$.
\end{proposition}
\begin{proof}
By symmetry and the strong Markov property,
it suffices to prove that $\tau_\infty<\infty$ a.s. only for $n=2$ and $(X_0,Y_0)=(x,0)$ for some $x\in G^*$. 
We use the notations of Section \ref{orbm2pt}.

Since $(|X_t|, t\le \tau_{\infty})$ is a reflected Brownian motion stopped at time $\tau_{\infty}$, it suffices to prove that a.s. $L_{\tau_{\infty}}(|X|)<\infty$. Denote by $L^{1}_t$ and $L^2_t$ the local times accumulated by
$Z$ respectively on $\{u=0\}$ and $\{v=0\}$ up to $t$ and
$L_t=L^1_t+L^2_t$.
First, note that for $t\le S_1$, $L_t(V^r)=(1-p_i)L_{\gamma(t)}(V)=(1-p_i)L_{\gamma(t)}(|Y|)$.
Thus $L_{\tau_1}(|Y|)=\frac{L_{S_1}(V^r)}{1-p_i}$. Note also that $L_{\tau_1}(|X|)=0$.
Thus $$L_{\tau_1}(|X|)+L_{\tau_1}(|Y|)=\frac{L_{S_1}}{1-p_i}.$$
Set $C=\sup_{\{1\le i\le N\}} (1-p_i)^{-1}$.
By induction, we get that
\begin{eqnarray*}
L_{\tau_\infty}(|X|)+L_{\tau_\infty}(|Y|) &=& \sum_{n\ge 0} \frac{L_{T_{n+1}} - L_{T_n}}{1-p_{i_n}} \;\le\; C\, L_{T_\infty}
\end{eqnarray*}
By Theorem \ref{hjh} a.s. $L_{T_{\infty}}<\infty$, and so $L_{\tau_{\infty}}(|X|)+L_{\tau_{\infty}}(|Y|)<\infty$.
\end{proof}

\subsection{Construction of $\varphi$}\label{constphi}
Let $(P^{(n)}, n\ge 1)$ be the unique consistent family of
Markovian semigroups such that
\begin{itemize}
\item $P^{(1)}$ is the semigroup of the WBM on
$G$.
\item The $n$-point motion of $P^{(n)}$ started from $x\in
G^n$ up to its entrance time in $\Delta_n$ is distributed as
$\mathbb P^{(n),0}_x$.
\item The $n$-point motion $(X^1,\dots,X^n)$ of $P^{(n)}$ is
such that if $X^i_s=X^j_s$ then $X^i_t=X^j_t$ for all $t\ge s$.
\end{itemize}
We will prove that $P^{(n)}$ is Feller for all $n$ and that (\ref{hila})
holds. By \cite[Lemma 1.11]{MR2060298}, this amounts to check the following condition.
\begin{lemma}
Let $(X,Y)$ be the two point motion associated to
$P^{(2)}$, then
for all positive $\eps$
$$\lim_{d(x,y)\to 0}\PP^{(2),0}_{(x,y)}[d(X_t,Y_t)>\eps] = 0.$$
\end{lemma}
\begin{proof}
As in the proof of Proposition \ref{coal}, we take $y=0$. Then using
the same notations, for all positive $\epsilon$,
$\{d(X_t,Y_t)>\eps\}\subset\{\sup_{t<\sigma_0}|Z_t|>\epsilon\}$.
Now the lemma follows from Remark \ref{incl}.
\end{proof}
By Theorem \ref{610}, a SFM $\varphi$ can be associated to
$(P^{(n)})_n$.
\begin{proposition}
Let $\varphi$ be a SFM associated to $(P^{(n)})_n$. Then
there exists a family of independent white noises $\mathcal
W=(W^i, 1\le i\le N)$ such that
\begin{itemize}\item[(i)] $\mathcal F^{\mathcal W}_{s,t}\subset
\mathcal F^{\varphi}_{s,t}$ for all $s\le t$ and 
\item[(ii)]$(\varphi,\mathcal W)$ solves $(\hbox{ISDE})$.
\end{itemize}
\end{proposition}
\begin{proof}
Let $V_{s,\cdot}(x)$ be the Brownian motion associated to
$\varphi_{s,\cdot}(x)$. For all $i\in[1,N]$ and $s\le t$, set
$$W^{i}_{s,t}=\lim_{|x|\to\infty, x\in E_i, |x|\in\QQ}
V_{s,t}(x).$$
For all $i\in[1,N]$ and $s\le t$, with probability $1$, this
limit exists.
Indeed if $x, y\in E_i$ are such that $|x|\le|y|$, then a.s.
$V_{s,t}(x)=V_{s,t}(y)$ for all $s\le t\le \tau^x_s=\inf\{u\ge
s;\; \p_{s,u}(x)=0\}$.
Moreover $W^i=(W^i_{s,t}, s<t)$ is a real white noise.
Indeed,  $W^i$ is centered and  Gaussian,
and by the flow property of $\varphi$ and using
$\varphi_{s,u}(x)=e_i(|x|+W^i_{s,u})$ if $s\le u\le \tau^x_s$ and
$x\in E_i$,
we have $W^i_{s,u}=W^i_{s,t}+W^i_{t,u}$. 
It is also clear that $W^i$ has independent increments with
respect to $(s,t)$. Thus, $W^i$ is a real white noise.
The fact that $\mathcal W=(W^i,1\le i\le N)$ is a family of
independent real white noises easily holds.

For $x\in G$ and $t\ge 0$, $$ \langle
W^i_{s,\cdot},V_{s,\cdot}(x)\rangle_t = \lim_{|y|\to\infty, y\in
E_i, |y|\in\QQ}\langle V_{s,\cdot}(y),V_{s,\cdot}(x)\rangle_t =
\int_{s}^{t}1_{\{\p_{s,u}(x)\in E_i\}}du.$$
This yields 
$$V_{s,t}(x)=\sum_{i=1}^N\int_{s}^{t}1_{\{\p_{s,u}(x)\in
E_i\}}dW^i_u.$$
By Theorem \ref{nh}, we deduce that $(\varphi,\mathcal W)$ solves
$(\hbox{ISDE})$.\end{proof}
Denote by $\PP_E$ the law of $(\varphi,\mathcal{W})$.

\subsection{Uniqueness in law of a SFM solution of
$(\hbox{ISDE})$}\label{eqq}
In this section, we show that the SFM $\varphi$ constructed in Section \ref{constphi} is the only SFM solution of $(\hbox{ISDE})$.
More precisely, we show
\begin{proposition}\label{uniquephi}
Let $(\varphi,\mathcal{W})$ be a solution of $(\hbox{ISDE})$, with
$\varphi$ a SFM. Then the law of $(\varphi,\mathcal{W})$ is
$\PP_E$.
\end{proposition}
\begin{proof}
We start by showing
\begin{lemma}\label{wsw0}
For all $x=e_i(r)\in G$, we have
$\varphi_{s,t}(x)=e_i(r+W^i_{s,t})$ for all $s\le t\le
\tau^x_s=\inf\{t\ge s : \varphi_{s,t}(x)=0\}$. In particular for
all $1\le i\le N$, $s\le t$, we have $\mathcal
F^{W^i}_{s,t}\subset \mathcal F^{\varphi}_{s,t}$.
\end{lemma}
\begin{proof}
Fix $i\in [1,N]$. Let $f\in \mathcal D$ be such that $f(x)=|x|$ for all $x\in E_i$. By
applying $f$ in $(\hbox{ISDE})$, we deduce the first claim. The second
claim is then an immediate consequence by taking a sequence
$(x_k)_k\subset E_i$ converging to $\infty$.
\end{proof}

With this lemma and Theorem \ref{propunic} we prove the
following
\begin{lemma}\label{lem:MPunicity} Let $x=(x_1,\cdots,x_n)\in
G^n$. Let $S=\inf\{t\ge 0 :
(\varphi_{0,t}(x_1),\cdots,\varphi_{0,t}(x_n))\in\Delta_n\}$.
Then $(\varphi_{0,t}(x_1),\cdots,\varphi_{0,t}(x_n))_{t\le S}$ is
distributed like $\PP^{(n),0}_x$. \end{lemma}
\begin{proof} 
Let $x\in G^n\setminus \Delta_n$ such that $x_{j_0}=0$ for some $j_0$.
For $j\in [1,n]$, set
$Y^j_t=\varphi_{0,t}(x_j)$ and $Y^{(n)}_t=(Y^1_t,\dots,Y^n_t)$.
Set for $i\in [1,N]$, $W^i_t=W^i_{0,t}$ and
$W_t=(W^1_t,\dots,W^n_t)$.
Note that for all $j\in [1,n]$, $(Y^j,W)$ is a solution of $(\hbox{ISDE})$. Set
$$\sigma_1=\inf\{t\ge 0 :\; \exists j\ne j_0 \hbox{ such that } Y^j_t=0\}$$
and for $\ell\ge 1$, 
$$\sigma_{ \ell+1}=\inf\{t\ge \sigma_\ell :\, \exists j
\hbox{ such that $Y^j_t=0$ and $Y^j_{\sigma_\ell}\neq 0$}\}.$$
Let $\sigma_\infty=\lim_{\ell\to\infty}\sigma_\ell$, then $\sigma_\infty=\inf\{t :
Y^{(n)}_t\in\Delta_n\}$.
By Theorem \ref{propunic}, the law of $(Y^{j_0},W)$ is uniquely
determined.
Now, for $j\ne j_0$, we have that for $t\le \sigma_1$, $Y^j_t=e_i(|x_j|+W^i_t)$ where $i$ is such that $x_j\in E_i^\ast$.
This shows that $(Y^{(n)}_t)_{t\le \sigma_1}$ is distributed as
$(X^{(n)}_t)_{t\le \tau_1}$, constructed in Subsection
\ref{qqq0}. Adapting the previous argument on the time interval
$[\sigma_\ell,\sigma_{\ell +1}]$, we show that for all $\ell\ge
1$, $(Y^{(n)}_t)_{t\le \sigma_\ell}$ is distributed as
$(X^{(n)}_t)_{t\le \tau_\ell}$. The Lemma easily follows.
\end{proof}

Lemma \ref{lem:MPunicity} permits to conclude the proof of
Proposition \ref{uniquephi}.
Indeed, the law of a SFM is uniquely determined by its family of
$n$-point motions $X^{(n)}$.
Using the fact that $\Delta_n$ is an absorbing set for $X^{(n)}$,
the strong Markov property at time $T^n=\inf\{t:\; X^{(n)}_t\in
\Delta_n\}$ and the consistency of the family of $n$-point
motions, we see that the law of a SFM is uniquely determined by
its family of $n$-point motions stopped at its first entrance
time in $\Delta_n$.
\end{proof}

\section{Extension to metric graphs}\label{last1}
Let $(I_k)_{1\le k\le K}$ be a family of finite sets such that $I_k\cap I_\ell \cap I_m=\emptyset$ for all $1\le k<\ell<m\le K$.
Let $(G_k)_{1\le k\le K}$ be a family of star graphs such that for each $k$, $G_k=\cup_{i\in I_k} E_k^i$, where $\{E_k^i,\;i\in I_k\}$ is the set of edges of $G_k$.
Set $I=\cup_{i=1}^K I_k$ and $\cI=\cup_{k=1}^K I_k\times \{k\}$.
For each $k$, let $p_k:=(p_k^i)_{i\in I_k}$ be a family of parameters such that $0<p_k^i<1$ and $\sum_{i\in I_k} p_k^i=1$. 

Let $\mathcal W=\{W^i_k,\; (i,k)\in \cI\}$ be a family of white noises such that 
\begin{itemize}
\item For all $k\ne \ell$ and $i\in I_k\cap I_\ell$, then $W^i_k+W^i_\ell=0$.
\item For all family $\{(i_1,k_1),\dots,(i_n,k_n)\}\subset \cI$ such that $i_1,\dots,i_n$ are $n$ distinct indices, $\{W^{i_1}_{k_1},\dots,  W^{i_n}_{k_n}\}$ is an independent family of white noises.
\end{itemize}
Notice that the law of $\mathcal{W}$ is uniquely described by these two properties and that $\mathcal{W}$ can be constructed out of a family of $|I|$ independent white noises.
Let us also remark that the second property implies that for each $k$, $\{W^i_k,\; i\in I_k\}$ is a family of independent white noises.

\bthm \label{idi-ind}
Let $(K_k)_{1\le k\le K}$ be a family of SFK's respectively defined on $G_k$. 
Assume that 
\begin{itemize}
\item For all $k$, $(K_k,W_k)$ solves $E(G_k,p_k)$, 
\item $(\cF_{s,t}:=\vee_k\cF^{K_k}_{s,t})_{s\le t}$ is i.d.i.
\end{itemize}
Then, the flows $(K_k)_{1\le k\le K}$ are independent given $\mathcal W$.
\ethm

The rest of this section will consist in proving Theorem \ref{idi-ind}.
For a SFK $K$ and a white noise $W$, $K(t)$ denotes $K_{0,t}$ and $W(t)$ denotes $W_{0,t}$. 

\subsection{Feller semigroups}\label{wdr}
Let $n:=(n_k)_{1\le k\le K}$ be a family of nonnegative integers and set $G^{(n)}:=\prod_k G_k^{n_k}$. For $t\ge 0$, $x:=(x_k)_{1\le k\le K}\in G^{(n)}$ and $w\in \RR^{|I|}$, set for $f\in C_0(G^{(n)})$ and $g\in C_0(\RR^{|I|})$,
$$\QQ^{(n)}_t (f\otimes g)(x,w)=\EE \big[\big(\otimes_k (K_k(t))^{\otimes n_k} \big) f (x) g(w+W(t))\big].$$
Note that the i.d.i property implies that $\QQ^{(n)}$ defines a Feller semigroup on $G^{(n)}\times \RR^{|I|}$. Denote by $\QQ^{(n)}_{(x,w)}$ the law of the diffusion started at $(x,w)$ associated to this semigroup.

Define also for all $k$, $\QQ^{(k,n_k)}_t$ the Feller semigroup on  $G_k^{n_k}\times \RR^{|I|}$ such that 
for $f_k\in C_0(G_k^{n_k})$ and $g\in C_0(\RR^{|I|})$,
$$\QQ^{(k,n_k)}_t (f_k\otimes g)(x_k,w)=\EE \big[\big(K_k(t))^{\otimes n_k} f_k\big) (x_k) g(w+W(t))\big].$$
Denote as above by $\QQ^{(k,n_k)}_{(x_k,w)}$ the law of the diffusion started at $(x_k,w)$ associated to this semigroup.

\medskip 
Let $(X,W)$ be a diffusion of law $\QQ^{(n)}_{(x,0)}$, then for all $(i,k)$, $(X^i_k,W)$ is a diffusion of law $\QQ^{(k,1)}_{(x_k^i,0)}$ and $(X^i_k,W_k)$ is a solution of $E(G_k,p_k)$  with $X^i_k(0)=x_k^i$.
This fact can easily be seen as a consequence of

\begin{lemma}
For all $k\in\{1,\dots,K\}$ and all $x\in G_k$, if $(X,W)$ is a diffusion of law $\QQ^{(k,1)}_{(x,0)}$, then $(X,W)$ is a solution of  $E(G_k,p_k)$.
\end{lemma}
\begin{proof} In the following, set $G=G_k$, $p=p_k$, $\QQ_t=\QQ^{(k,1)}_t$ and $N=|I_k|$.
It is obvious that $W$ is an $N$-dimensional Brownian motion. It is also clear that $X$ is a WBM. Denote by $B^X$ the Brownian motion associated to $X$. Then by Freidlin-Sheu formula, $(X,W)$ solves $E(G,p)$ as soon as $B^X_t=\sum_i \int_0^t 1_{\{X_s\in E_i\}} dW^i_s$. It is enough to prove that $\langle B^X,W^i\rangle_t = \int_0^t 1_{\{X_s \in E_i\}} ds$ for all $i$.

Recall the definition of $\mathcal D$ from (\ref{domain}) and set  
$\mathcal D_1=\{f\in\mathcal D : \, f, f', f''\in C_0(G)\}.$
Denote by $A$ the generator of $\QQ_t$ and $\mathcal D(A)$ its domain, then $\mathcal D_1\otimes C^2_0(\RR^{N})\subset \mathcal D(A)$ and for all $f\in\mathcal D_1$ and $g\in C^2_0(\RR^{N})$,
$$A(f\otimes g)(x,w)=\frac{1}{2}f(x) \Delta g(w) +\frac{1}{2}f''(x) g(w) + \sum_{i=1}^{N} (f'1_{E_i})(x)\frac{\partial g}{\partial w^i}(w).$$
Thus for all $f\in\mathcal D_1$ and $g\in C^2_0(\RR^{N})$,
\begin{equation}\label{mart}
f(X_t) g(W_t)-\int_{0}^{t} A(f\otimes g)(X_s,W_s) ds\ \text{is a martingale}.
\end{equation}
Applying Freidlin-Sheu formula for $f(X_t)$, and then It\^o's formula for $f(X_t) g(W_t)$, we get that, using \eqref{mart}
$$\sum_{i=1}^N \int_0^t (f'1_{E_i})(X_s)\frac{\partial g}{\partial w^i}(W_s)ds
=\sum_{i=1}^N \int_0^t (f'1_{E_i})(X_s)\frac{\partial g}{\partial w^i}(W_s)d\langle B^X,W_i\rangle_s.$$
Since this holds for all $f\in\mathcal{D}_1$ and $g\in C_0^2(\RR^N)$, we get
$\langle B^X,W^i\rangle_t = \int_0^t 1_{\{X_s \in E_i\}} ds$ for all $i$.
\end{proof}

\subsection{A sufficient condition for conditional independence}

For $x=(x_k)_k\in G^{(n)}$ and $w\in \RR^{|I|}$, let $\PP^{(n)}_{(x,w)}$ be the law of $(X_1,\dots,X_K,W)$ such that 
$X_1,\dots, X_K$ are independent given $W$ and for all $k$, $(X_k,W)$ is distributed as $\QQ^{(k,n_k)}_{(x_k,w)}$.
Denote by $\EE^{(n)}_{(x,w)}$ the expectation with respect to $\PP^{(n)}_{(x,w)}$. 
Denote also by $\EE^{(k,n_k)}_{(x_k,w)}$ the expectation with respect to $\PP^{(k,n_k)}_{(x_k,w)}$.  
For $Z$, a $\sigma(W)$-measurable random variable, we simply denote $\EE^{(n)}_{(x,0)}[Z]$ and $\EE^{(k,n_k)}_{(x_k,w)}[Z]$  by $\EE[Z]$.

\bprop \label{cond-ind}
If for all $n:=(n_k)_k$ and all $x:=(x_k)_k\in G^{(n)}$, 
\begin{equation}\label{condd}
\QQ^{(n)}_{(x,0)}=\PP^{(n)}_{(x,0)}, 
\end{equation}
then the flows $(K_k)_k$ are independent given $W$.
\eprop
\begin{proof}
Fix $n=(n_k)_k$, $x=(x_k)_k$ and $f=\otimes_k f_k$, with $f_k\in C_0(G_k^{n_k})$, then
\begin{eqnarray*}
\EE \big[\prod_k (K_k(t))^{\otimes n_k} f_k(x_k)\big]
&=& \QQ^{(n)}_t  (f\otimes 1) (x,0) \\
&=& \EE^{(n)}_{(x,0)} \big[\prod_k f_k(X_k(t))\big]\\
&=& \EE \big[ \prod_k  \EE_{(x_k,0)}^{(k,n_k)}[f_k(X_k(t))|W] \big].
\end{eqnarray*}
Then the proposition follows from the fact that
\beq \EE_{(x_k,0)}^{(k,n_k)}[f_k(X_k(t))|W]  = \EE[(K_k(t))^{\otimes n_k}f_k(x_k)|W]. \label{knksachw}\eeq
Let us check \eqref{knksachw}. For this, take $g_0,\cdots,g_J$ in $C_0(\RR^{|I|})$ and fix $0=t_0<\cdots<t_J=t$. 
Then setting, for all $g\in C_0(\RR^{|I|})$, $h\in C_0(G_k^{n_k}\times \RR^{|I|})$ and all $t\ge 0$,
$\QQ^g_th(x,w)=g(w)\QQ^{(k,n_k)}_t h (x,w)$, one has (to lighten the notation below, $f_k$, $x_k$, $X_k$ and $(K_k(t))^{\otimes n_k}$ are denoted by $f$, $x$, $X$ and $K_t$)
\begin{eqnarray*}
\EE \big[K_t f(x) \prod_{0\le j\le J} g_j(W(t_j))\big]
&=& \QQ^{g_0}_{t_1}\cdots  \QQ^{g_{J-1}}_{t_J-t_{J-1}}(f\otimes g_J) (x,0)\\
&=& \EE^{(k,n_k)}_{(x,0)} \big[f_k(X(t)) \prod_{0\le j\le J} g_j(W(t_j))\big]\\
&=& \EE \big[ \EE_{(x,0)}^{(k,n_k)}[f(X(t))|W]  \prod_{0\le j\le J} g_j(W(t_j)) \big],
\end{eqnarray*}
which suffices to deduce \eqref{knksachw}.
\end{proof}

Note that, using the Feller property, \eqref{condd} is satisfied for all $n$ and all $x$ as soon as it is satisfied for all $n$ and all $x$ in a dense subset of $G^{(n)}$.

\subsection{Uniqueness up to the first meeting time at $0$} \label{umt0}
A point $x_k\in G_k^{n_k}$, will be denoted by $(x^i_k)_{1\le i\le n_k}$. Take $n=(n_k)_k$  as in Section \ref{wdr} and choose $x=(x_k)_k\in G^{(n)}$ such that for all $k\ne \ell$, $0\not\in \{x^i_k, 1\le i\le n_k\}\cap \{x^j_l, 1\le j\le n_l\}$. 
In the following, $(X,W)$ will be distributed as $\PP^{(n)}_{(x,0)}$ or as $\QQ^{(n)}_{(x,0)}$ with $X=(X_1,\cdots,X_K)$, and $(\cF_t)_{t\ge 0}$ denotes the filtration generated by $(X,W)$.
For $t\ge 0$, let $R_k(t):=\{X^i_k(t)\; | \; 1\le i\le n_k\}$. 
Define the sequence of stopping times $(\sigma_j)_{j\ge 0}$ such that $\sigma_0=0$ and for all $j\ge 0$,
\beq \sigma_{j+1}=\inf\{t\ge \sigma_j\; | \; \exists k,\; 0\in R_k(t) \hbox{ and } 0\not\in R_k(\sigma_j)\}. \eeq
Using the strong Markov property, it is easy to see that for all $j\ge 1$, there is only one $k$ such that $0\in R_k(\sigma_j)$
and that the sequence $(\sigma_j)_{j\ge 1}$ is strictly increasing. Denote by $\sigma_\infty=\lim_{j\to\infty}\sigma_j$.
\bprop \label{unicupcoal}
The law of $(X(t),W(t))_{t<\sigma_\infty}$ is the same under $\QQ^{(n)}_{(x,0)}$ and  under $\PP^{(n)}_{(x,0)}$.
\eprop
\begin{proof} Let $(X,W)$ be distributed as $\QQ^{(n)}_{(x,0)}$.
Without loss of generality, assume there exists
$\ell$ such that $0\in R_\ell(0)$ and $0\not\in \cup_{k\ne \ell} R_k(0)$. 
Then $\{(X_k(t), t\le \sigma_{1}),\, k\ne\ell\}$ is $\sigma(W)$-measurable and
therefore $\{(X_k(t), t\le \sigma_{1}),\, 1\le k \le K\}$ is a family of independent random variables given $W$.
So the conditional law of this family given $W$ is the same as the conditional law of  $\{(X_k(t), t\le \sigma_{1}),\, 1\le k \le K\}$ given $W$ whenever $(X,W)$ is distributed as $\PP^{(n)}_{(x,0)}$.
Denote this law by $\mu(x,W)$.
Using the strong Markov property at time $\sigma_n$, we get that given $\cF_{\sigma_n}$ and $W$, the law of $\{(X_k(t+\sigma_n), t\le \sigma_{n+1}-\sigma_n),\, 1\le k \le K\}$ is  $\mu(X(\sigma_n),W(\cdot+\sigma_n) - W(\sigma_n))$. Since this characterizes the law of $(X(t),W(t))_{t<\sigma_\infty}$, the proposition is proved.
\end{proof}

This proposition implies in particular that if $\PP^{(n)}_{(x,0)}(\sigma_\infty=\infty)=1$, 
then $\QQ^{(n)}_{(x,0)}=\PP^{(n)}_{(x,0)}$.

\subsection{The meeting time at $0$ is infinite} Our purpose here is to prove the following
\bprop \label{siginftyfini}
For all $n=(n_k)$ and  $x=(x_k)\in G^{(n)}$ such that $0\not\in \{x^i_k, 1\le i\le n_k\}\cap \{x^j_l, 1\le j\le n_l\}$, for all $k\ne \ell$, we have
\begin{equation}\label{eq:pn0} \PP^{(n)}_{(x,0)}(\sigma_\infty=\infty)=1.\end{equation}
\eprop
\begin{proof}
Assume $K=2$, $n_1=n_2=1$ and take $x=(x_1,0)$ with $x_1\ne 0$.
It is easy to see that if \eqref{eq:pn0} holds in this particular case, then it also holds in the general case.
We use in the following the notations of Subsection \ref{umt0}.
Note that for $k\in\{1,2\}$, $(X_k,W_k)$ is a solution of $E(G_k,p_k)$ and that
$$\sigma_\infty=\inf\{t\ge 0 : X_1(t)=X_2(t)=0\}.$$  
Set $I^c=I_1\cap I_2$.  For $k\in\{1,2\}$ and $i\in I_k$, set $\theta^i_k=\arctan\big(\frac{p^{i}_{k}}{1-p^{i}_{k}}\big)$ if $i\in I^c$ and $\theta^i_k=0$ otherwise.
Recall the definition of $(\sigma_n)_n$\,:  $\sigma_0=0$ and for all $\ell\ge 0$,
\begin{eqnarray*}
\sigma_{2\ell+1} &=& \inf\{t\ge \sigma_{2\ell} : X_1(t)=0\},\\
\sigma_{2\ell+2} &=& \inf\{t\ge \sigma_{2\ell+1} : X_2(t)=0\}.
\end{eqnarray*}
For all $\ell\ge 0$, set $U^0_{2\ell}=|X_1(\sigma_{2\ell})|$ and $U^0_{2\ell+1}=|X_2(\sigma_{2\ell+1})|$. Let $i_{2\ell}\in I_1$ and $i_{2\ell+1}\in I_2$ be such that $X_1(\sigma_{2\ell})\in E_1^{i_{2\ell}}$ and $X_2(\sigma_{2\ell+1})\in E_2^{i_{2\ell+1}}$. 
Set $\Theta^n=-\theta_2^{i_n}$ if $n$ is even and $\Theta^n=-\theta_1^{i_n}$ if $n$ is odd.

Say $x_1\sim x_2$ if there is $i\in I^c$ such that $x_1\in E_1^i$ and $x_2\in E_2^i$ and say $x_1\not\sim x_2$ otherwise.  Set 
$A_t=\int_0^t 1_{\{ X_1(s)\not\sim X_2(s)\}}ds$, 
$\gamma_t=\inf\{s\ge 0 : A_s>t\}$ and $\mathcal G_t={\mathcal F}_{\gamma_t}.$
Set $S_n=A_{\sigma_n}$ and define $\Theta_t=\Theta^{n}$ for $t\in [S_{n},S_{n+1}[$. Note that $S_\infty:=\lim_{n\to \infty} S_n=A_{\sigma_\infty}$.

Our purpose now is to define a c\`adl\`ag process $(U_t,V_t)_{t< S_\infty}$, 
such that for all $n\ge 0$, conditionally on $\cG_{S_{n}}$, $(U_t,V_t)_{t\in [S_{n},S_{n+1})}$ is a Brownian motion in the quadrant $\cQ$, obliquely reflected on $\partial_1\cQ$ with angle of reflection $\Theta^n$, such that 
\begin{itemize}
\item $(U_{S_{n}},V_{S_{n}})=(U^0_{n},0)$;
\item  $U_t>0$ for all $t\in [S_{n},S_{n+1})$
\end{itemize}
and such that (using the notation $H_{t-}=\lim_{s\uparrow t} H_s$ for $H$ a c\`adl\`ag process)
\begin{itemize}
\item when $n$ is odd, $V_{S_{n+1}-}=0$ if $X_1(\sigma_{n+1})\in E^{i_{n}}_1$ and $U_{S_{n+1}-}=0$ if not,
\item when $n$ is even, $V_{S_{n+1}-}=0$ if $X_2(\sigma_{n+1})\in E^{i_{n}}_2$ and $U_{S_{n+1}-}=0$ if not.
\end{itemize}

The construction being exactly the same for all $n\ge 0$, we just do it for $n=0$. 
Note that $X_{1}(t)\in E_1^{i_0}$ and $|X_1(t)|=|x_1|+W_1^{{i_0}}(t)$ for $t\le \sigma_1$. We then have two cases\,:\\
{\bf First case\,:} $i_0\in  I_1\setminus I^c$. For all $t<\sigma_1$, we have $X_2(t)\not\sim X_1(t)$ and $A(t)=t$. We also have $S_1=\sigma_1$.
Set for $t < S_1$, $U_t=|X_1(t)|$ and $V_t=|X_2(t)|$. 
Then $(U_t,V_t)_{t<S_1}$ is a normally reflected Brownian motion in the quadrant $\mathcal{Q}$ started at $(|x_1|,0)$ and killed when hitting $\{x=0\}$.\\
{\bf Second case\,:} $i_0\in I^c$. For $t\le \sigma_1$, set $X_t=|X_1(t)|$ and
$$Y_t=|X_2(t)| \big(1_{\{X_2(t)\not\sim X_1(t)\}} - 1_{\{X_2(t)\sim X_1(t)\}}\big).$$
Note that for $t\leq \sigma_1$, $A_t=\int_0^t 1_{\{Y_s>0\}}ds$ and $1_{\{Y_t<0\}}d(X_t+Y_t)=0$. 
The process $(X_t,Y_t)_{t\le \sigma_1}$ behaves as a two dimensional Brownian motion in $\mathcal Q$ and evolves on straight lines parallel to $\{y=x\}$ outside $\mathcal Q$  .
Finally set $\big(U_t,V_t\big)=\big(X_{\gamma_t},Y_{\gamma_t}\big)$, for $t< S_1$.
Following the proof of Lemma \ref{tyjh}, we check that $(U_t,V_t)_{t<S_1}$ is an obliquely reflected Brownian motion, with angle of reflection $\Theta^0$.

Define the sequence $(T_j)_{j\ge 0}$ by $T_0=0$ and for all $j\ge 0$,
$$T_{j+1}=\inf\{S_n : \; S_n>T_j\ \text{and}\ U_{S_n-}=0\}.$$
Set $T_\infty=\lim_{j\rightarrow\infty} T_j=S_\infty=A_{\sigma_\infty}$. Then $(U_t,V_t)_{t<S_{\infty}}$ is continuous except at the times $T_j$, $j\ge 1$.
Moreover, we have the following
\begin{lemma}
The process $(U_t,V_t)$ is $(\mathcal G_t)_t$-adapted.
There is a two-dimensional $(\mathcal G_t)_t$-Brownian motion $(B^1,B^2)$, such that for all $j\ge 0$ and all $t\in [T_{j},T_{j+1}[$, 
\begin{eqnarray*}
U_t &=& U_{T_{j}} + \int_{T_{j}}^t \big(dB^1_s-\tan\big(\Theta_s\big) dL_s(V)\big),\\
V_t &=&  \int_{T_{j}}^t \big(dB^2_s+dL_s(V)\big).
\end{eqnarray*}
Moreover $U_{T_{j+1}}=V_{T_{j+1}-}$,  $U_{T_{j+1}-}=0$ and $U_t>0$ for all $t<S_\infty$.
\end{lemma}

Out of the noncontinuous process $(U,V)$, we construct a continuous process $(Z^r_t=(X^r_t,Y^r_t),t<T_{\infty})$ by\,:
For all $n\ge 0$,  $(X^r_t,Y^r_t)=(U_t,V_t)$ for $t\in [T_{2n},T_{2n+1})$ and $(X^r_t,Y^r_t)=(V_t,U_t)$ for $t\in [T_{2n+1},T_{2n+2}[$. Then, putting everything together, we get the following

\begin{lemma}
The process $(Z^r_t,t<T_{\infty})$ is an obliquely reflected Brownian motion on $\mathcal{Q}$ started from $(|x_1|,0)$, and we have, for all $n\ge 0$,
\begin{eqnarray*}
T_{2n+1}&=&\inf\{t\ge T_{2n} : Y^r_t=0\},\\
T_{2n+2}&=&\inf\{t\ge T_{2n+1} : X^r_t=0\},
\end{eqnarray*}
$\lim_{t\uparrow T_\infty} Z_t^r=(0,0)$ and for all $t<T_{\infty}$,
\begin{eqnarray*}
dX^r_t &=& dB^1_t+dL_t({X^r})-\tan\big( \Theta_t\big) dL_t(Y^r),\\
dY^r_t &=& dB^2_t-\tan\big( \Theta_t\big) dL_t({X^r})+dL_t(Y^r).
\end{eqnarray*}
Moreover $T_\infty=\infty$ implies $\sigma_\infty=\infty$.
\end{lemma}

To conclude the proof of Proposition \ref{siginftyfini}, it remains to prove that a.s. $T_{\infty}=\infty$. We exactly follow \cite[Page 161]{MR1121940}. For $a\ge 0$, define 
$$\tau_{a}=\inf\{t\ge 0 : |Z^r_t|=a\}.$$
Take $\epsilon<|x_1|<A$ and set $\tau_{\epsilon,A}=\tau_{\epsilon}\wedge\tau_A$. Then, by It\^o's formula, setting $R_t=|Z^r_t|$, we have that, for all $t\ge 0$,
$$\log(R_{t\wedge\tau_{\epsilon,A}})=\log(|x_1|)+M_t+C_t$$
where $M$ is a martingale started from $0$ and $C$ is a nonnegative nondecreasing process (using that $\Theta_t\le 0$). 
Thus by letting $t\to\infty$, we get $E[\log(R_{\tau_{\epsilon,A}})]\ge \log(|x_1|)$. So
$$\log(\epsilon) \mathbb P(\tau_{\epsilon}<\tau_{A}) + \log(A) (1-\mathbb P(\tau_{\epsilon}<\tau_{A})) \ge \log(|x_1|)$$
and consequently
$$\mathbb P(\tau_{\epsilon}<\tau_{A})\le \frac{\log(A)-\log(|x_1|)}{\log(A)-\log(\epsilon)}.$$
Replacing $\epsilon$ with $\epsilon(A)=A^{-A}$, yields
$$\mathbb P(T_{\infty}<\infty)=\lim_{A\rightarrow\infty} \mathbb P(T_{\infty}<\tau_{A})\le \lim_{A\rightarrow\infty}\mathbb P(\tau_{\epsilon(A)}<\tau_{A})=0.$$
\end{proof}
\section{Final remarks}\label{final}
It would be interesting to extend the framework of the present paper to the case of a star graphs with an infinite number of rays $G=\cup_{n\in \NN} E_n$. Suppose we are given a family $p=(p_n)_{n\in \NN}\subset ]0,1[$ such that $\sum_n p_n=1$. Then the WBM associated to $p$ can still be defined via its semigroup (as in the introduction). 
It satisfies also a Freidlin-Sheu formula similar to the finite case (see \cite{MR501878}):
$$df(Z_t)=f'(Z_t)dB^Z_t + \frac{1}{2} f''(Z_t) dt$$
where $B^Z$ is again the martingale part of $|Z|$ and $f$ runs over an appropriate domain of  functions $\mathcal D$. 
Now suppose given a family $(W^n)_{n\in\NN}$ of independent Brownian motions, then the natural extension of $(\hbox{ISDE})$ associated to $p$ is the following
$$df(Z_t)=\sum_n (f'1_{E_n}) (Z_t) dW^n_t + \frac{1}{2} f''(Z_t) dt,\ f\in\mathcal D$$
which we denote again by $(\hbox{ISDE})$. The Brownian motion $B^Z$ has also the martingale representation property for $(\mathcal F^Z)_t$ \cite[Proposition 19 (ii)]{MR1655299}. Thus following our arguments, under some conditions on $Z$, the law of any solution $(Z,W^n, n\in\NN)$ to $(\hbox{ISDE})$ is unique. One could also investigate stochastic flows solutions of $(\hbox{ISDE})$. However, in contrast to the finite case, here we have 
$$\inf\big\{\arctan\big(\frac{p_n}{1-p_n}\big) : n\in\NN\big\}=0.$$
This is the new difficulty with respect to the present paper. We leave the question of existence of a SFM in this case open.

\medskip
\textbf{Acknowledgement.} We are grateful to Michel \'Emery for
very useful discussions.

\bibliographystyle{plain}
\bibliography{Bil6}
\end{document}